\documentclass{amsart}

\usepackage{amsmath,amssymb,amsthm,enumerate,bm}

\usepackage{ascmac}

\usepackage[margin=40truemm]{geometry}

\usepackage[dvipdfmx]{color}
\usepackage{graphicx}
\usepackage{tikz}
\usetikzlibrary{angles,calc}
\usetikzlibrary {arrows.meta}

\theoremstyle{definition}
\newtheorem{dfn}{Definition}[section]
\newtheorem{thm}[dfn]{Theorem}
\newtheorem{lem}[dfn]{Lemma}
\newtheorem{cor}[dfn]{Corollary}
\newtheorem{prop}[dfn]{Proposition}

\theoremstyle{remark}
\newtheorem{rem}[dfn]{Remark}

\newcommand{\lrangle}[1]{\langle #1 \rangle}
\renewcommand{\Im}{\mathrm{Im}}
\newcommand{\Newt}{\mathrm{Newt}}
\newcommand{\conv}{\mathrm{conv}}
\newcommand{\norm}[1]{\left\lVert #1 \right\rVert}
\newcommand{\Aff}{\mathrm{Aff}}

\newcommand{\id}{\mathrm{id}}

\newcommand{\matac}[3]{\begin{pmatrix} #1 & #2 & #3 \end{pmatrix}}
\newcommand{\matca}[3]{\begin{pmatrix} #1 \\ #2 \\ #3 \end{pmatrix}}

\newcommand{\Txpm}{\mathbb T[\bm x^{\pm}]}

\newcommand{\Txpmf}{\mathbb T[\bm x^{\pm}]_{\mathrm{fcn}}}
\newcommand{\Typmf}{\mathbb T[\bm y^{\pm}]_{\mathrm{fcn}}}

\makeatletter
\@namedef{subjclassname@2020}{%
  \textup{2020} Mathematics Subject Classification}
\makeatother

\begin{document}

\title{Non-finite generatedness of the congruences defined by tropical varieties}
\author{Takaaki Ito}
\address{Academic Support Center, Kogakuin University, 2665-1 Nakano-machi, Hachioji-shi, Tokyo, 192-0015, Japan}
\email{kt13764@ns.kogakuin.ac.jp}

\begin{abstract}
  In tropical geometry, there are several important classes of ideals and congruences such as tropical ideals, bend congruences, and the congruences of the form $\mathbf E(Z)$.
  Although they are analogues of the concept of ideals of rings, it is not well known whether they are finitely generated.
  In this paper, we study whether the congruences of the form $\mathbf E(Z)$ are finitely generated.
  In particular, we show that when $Z$ is the support of a tropical variety, $\mathbf E(Z)$ is not finitely generated except for a few specific cases.
  In addition, we give an explicit minimal generating set of $\mathbf E(|L|)$ for the tropical standard line $L$.
\end{abstract}

\maketitle

\section{Introduction}

Tropical geometry is a subject of mathematics developed in recent decades, which studies geometric objects called tropical varieties.
Sometimes tropical geometry is called \textquotedblleft the algebraic geometry over $\mathbb T$\textquotedblright, where $\mathbb T$ is a semifield called the tropical semifield.
Since $\mathbb T$ is not a field, tropical geometry is not included in classical algebraic geometry.
However, for many albedo-geometric theories, their tropical versions are studied.
For instance, the tropical intersection theory \cite{allermann2010first,jensen2013computing,jensen2016stable,katz2012tropical,rau2016intersections}, the theory of divisors on tropical curves \cite{amini2013reduced,an2014canonical,haase2012linear,hladky2013rank,len2014brill}, and the theory of tropical homology \cite{amini2020hodge,gross2019sheaf,itenberg2019tropical,ruddat2021homology} have been developed.

One of the basic objects in algebraic geometry is the ideals of rings.
In tropical geometry, there are several analogous objects.
The simplest one is the concept of ideal of a semiring, which is defined as a submodule of the semiring itself.
Another one is the concept of tropical ideal, which is defined by Maclagan and Rinc\'{o}n in \cite{maclagan2018tropical}.
It is an analogue of an ideal of a ring since the quotient semiring $R/E$ is defined.
Moreover, there are several specific class of congruences.
In \cite{giansiracusa2016equations}, J. Giansiracusa and N. Giansiracusa define \textit{bend congruences} to develop the theory of scheme theoretic tropicalization.
Another important class is the congruences of the form $\mathbf{E}(Z)$ for some $Z \subset \mathbb R^n$.
Here,
$$\mathbf{E}(Z) = \{ (f,g) \in \mathbb T[x_1, \ldots, x_n] \ | \ f|_Z = g|_Z \}.$$
The semiring $\mathbb T[x_1, \ldots, x_n]$ is sometimes replaced by another one such as the Laurent polynomial semiring $\mathbb T[x_1^{\pm}, \ldots, x_n^{\pm}]$, the rational function semifield $\mathbb T(x_1, \ldots, x_n)$, the polynomial ring $\mathbb B[x_1, \ldots, x_n]$ over the Boolean semifield, or the polynomial function semiring $\mathbb T[x_1, \ldots, x_n]_{\mathrm{fcn}}$.
This is a simple analogue of the concept of vanishing ideal in algebraic geometry.
The congruences of this form are studied by Bertram and Easton in \cite{bertram2017tropical}, Jo\'{o} and Mincheva in \cite{joo2018prime}, Song in \cite{song2024congruences}, and the author in \cite{ito2022local}.
Also, \textit{prime congruence} are defined in \cite{joo2018prime}, which are of course analogues of the prime ideals of rings.

Almost all rings appeared in classical algebraic geometry are Noetherian, and their ideals are finitely generated.
This fact is used to develop the dimension theory.
Hence it is natural that we consider whether a given ideal (or congruence) is finitely generated.
There are few known results.
In \cite[Example 3.10]{maclagan2018tropical}, an example of tropical ideal which is not finitely generated is given.
Also, it holds by definition that a bend relation defined by a polynomial is finitely generated.
In this paper, we consider the congruences on $\mathbb T[x_1^{\pm}, \ldots, x_n^{\pm}]$ of the form $\mathbf E(Z)$.
In particular, we consider the case that $Z$ is the support of a tropical variety.
The following is our main result.

\begin{thm}
  \label{main1}
  Let $X$ be a tropical variety in $\mathbb R^n$.
  Then $\mathbf E(|X|)$ is finitely generated if and only if $|X|$ is a rational affine linear subspace of $\mathbb R^n$.
\end{thm}

Thus, for almost all tropical variety $X$, $\mathbf E(X)$ is not finitely generated.
The only if part of this theorem is a consequence of the following more general result.

\begin{thm}
  \label{main2}
  Assume that $n \geq 2$.
  Let $Z$ be a proper closed subset of $\mathbb R^n$ that is unbounded in all rational directions.
  Then $\mathbf E(Z)$ is not finitely generated.
\end{thm}

The strategy of the proof of Theorem \ref{main2} is as follows.
First we assume that $\mathbf E(Z)$ is finitely generated and fix a generating set.
Then we construct a pair $(f,g) \in \mathbf E(Z)$ such that $f \neq g$ and the distance of any two lattice points of the Newton polytope of $f$ is sufficiently large.
The existence of such pair causes a contradiction.
See Section 3 for detail.

Note that we cannot apply Theorem \ref{main2} for showing the only if part of Theorem \ref{main1} directly because $|X|$ is not unbounded in all rational directions in general.
In particular, if $|X|$ is included in a proper rational affine linear subspace of $\mathbb R^n$, then $|X|$ is not unbounded in all rational directions.
To overcome this issue, we show that we may replace the ambient space with a lower dimensional one.
See Section 4 for detail.

Finally, we give an example of minimal generating set of $\mathbf E(L)$ for the standard tropical line $L \subset \mathbb R^2$.
See Section \ref{section min gen} for detail.

\begin{thm}
  \label{main3}
  The set
  $$\mathcal S = \left\{ \left( \overline{0 \oplus x^u y^v}, f_{\Delta(\{ (0,0), (u,v) \})} \right) \ \middle| \ \begin{aligned}
  &\text{$u$ and $v$ are integers with $\gcd(u,v) = 1$,}\\
  &\text{ $u > 0$ or $(u,v) = (0,1)$}
  \end{aligned} \right\}$$
  is a minimal generating set of $\mathbf E(L)$.
\end{thm}

This paper is organized as follows.
In Section 2, we recall some basic definitions and facts about semirings, tropical algebra, tropical varieties, and Newton polytopes.
Also we show some lemmas we need later.
In Section 3, we consider a way to prove that a given congruence on $\mathbb T[x_1^{\pm}, \ldots, x_n^{\pm}]$ is not finitely generated, and use it to show Theorem \ref{main2}.
In Section 4, we discuss about the reduction of the ambient spaces.
The result of this section is necessary for showing Theorem \ref{main1}.
In Section 5, we prove Theorem \ref{main1}.
In Section 6, we give an example of minimal generating set of $\mathbf E(L)$ for the standard tropical line $L$.
This section is almost independent of Section 3-5.
In Appendix, we give proofs of some facts about free abelian groups and polytopes, which are used in this paper.

\section{Preliminary}

\subsection{Semirings}

A \textit{semiring} is a set $R$ endowed with two binary operations $+, \cdot$ called the \textit{addition} and the \textit{multiplication} respectively and a fixed element $0 \in R$ such that $(R, +, 0)$ is a commutative monoid, $(R, \cdot)$ is a semigroup, the multiplication is distributive over the addition, and $a \cdot 0 = 0 \cdot a = 0$ for any $a \in R$.
A semiring $R$ is \textit{unitary} if there exists $1 \in R$ such that $a \cdot 1 = 1 \cdot a = a$ for any $a \in R$.
A semiring $R$ is \textit{commutative} if $a \cdot b = b \cdot a$ for any $a,b \in R$.
In this paper, we assume that any semiring is unitary and commutative.
We often denote $ab$ instead of $a \cdot b$.
A semiring $R$ is called \textit{semifield} if $0 \neq 1$ and for any $a \in R \setminus \{ 0 \}$, there exists $b \in R$ such that $ab = 1$.
A \textit{subsemiring} of $R$ is a subset $S$ such that $0,1 \in S$ and, for any $a,b \in S$, $a+b \in S$ and $ab \in S$.

A \textit{congruence} on $R$ is an equivalence relation $E \subset R \times R$ on $R$ such that, if $(a,b), (c,d) \in E$, then $(a+c, b+d), (ac, bd) \in E$.
For a congruence $E$ on $R$, the \textit{quotient semiring} $R/E$ is naturally defined.
For a subset $S \subset R \times R$ there exists the minimum congruence on $R$ including $S$, which is called the congruence \textit{generated by} $S$, and denoted by $\lrangle{S}$.
If $S$ is a finite set, say $S = \{ (a_1,b_1), \ldots, (a_r, b_r) \}$, we denote
$$\lrangle{S} =\lrangle{ (a_1,b_1), \ldots, (a_r, b_r) }.$$
If a congruence is generated by a finite set, we say that the congruence is \textit{finitely generated}.

Let $R_1, R_2$ be semirings.
A map $\varphi : R_1 \to R_2$ is called a \textit{homomorphism} if $\varphi(a+b) = \varphi(a) + \varphi(b)$ for any $a,b \in R_1$, $\varphi(ab) = \varphi(a) \varphi(b)$ for any $a,b \in R_1$, $\varphi(0) = 0$, and $\varphi(1) = 1$.
The set $\Im(\varphi):=\{ \varphi(a) \in R_2 \ | \ a \in R_1 \}$ is called the \textit{image} of $\varphi$, which is a subsemiring of $R_2$.
A homomorphism $\varphi : R_1 \to R_2$ is called an \textit{isomorphism} if there exists a homomorphism $\psi : R_2 \to R_1$ such that $\psi \circ \varphi = \id_{R_1}$ and $\varphi \circ \psi = \id_{R_2}$.
We say that $R_1$ and $R_2$ are \textit{isomorphic} if there exists an isomorphism $\varphi : R_1 \to R_2$.

\subsection{Tropical algebra}
\label{pre trop}

The \textit{tropical semifield} $\mathbb T$ is the set $\mathbb R \cup \{ -\infty \}$ endowed with the addition $a \oplus b = \max \{ a,b \}$ and the multiplication $a \odot b = a+b$, which forms a semifield.
The \textit{Laurent tropical polynomial semiring} $\Txpm = \mathbb T[x_1^{\pm}, \ldots, x_n^{\pm}]$ is naturally defined.
We assume that $\Txpm$ always has $n$ variables unless otherwise mentioned.
Each tropical Laurent polynomial in $\Txpm$ defines a $\mathbb T$-valued function on $\mathbb R^n$.
We denote $\overline{P}$ the function defined by $P \in \Txpm$.
Also we denote $\Txpmf$ the set of functions defined by some $P \in \Txpm$, which forms a semiring.

We consider the pull-backs and translations of functions.

\begin{lem}
  \label{pull-buck}
  Let $\mu: \mathbb R^n \to \mathbb R^m$ be a linear map defined by an integer matrix.
  Then the \textit{pull-back} map
  $$\mu^* : \mathbb T[y_1^{\pm}, \ldots, y_m^{\pm}]_{\mathrm{fcn}} \to \mathbb T[x_1^{\pm}, \ldots, x_n^{\pm}]_{\mathrm{fcn}}, \qquad f \mapsto f \circ \mu$$
  is well-defined.
  Namely, for any $f \in \mathbb T[y_1^{\pm}, \ldots, y_m^{\pm}]_{\mathrm{fcn}}$, the function $f \circ \mu$ is always defined by some tropical Laurent polynomial.
  Moreover, $\mu^*$ is a semiring homomorphism.
\end{lem}

\begin{proof}
  The semiring $\Txpmf$ is a subsemiring of the semiring $\mathrm{Map}(\mathbb R^n, \mathbb T)$, the set of $\mathbb T$-valued functions on $\mathbb R^n$ endowed with the pointwise addition and the pointwise multiplication.
  The pull-back map $\mu^* : \Typmf \to \mathrm{Map(\mathbb R^n, \mathbb T)}$ is defined, and this is clearly a semiring homomorphism.
  Thus it suffices to show that $\Im(\mu^*) \subset \Typmf$.
  Moreover, it is enough to show that $\mu^*(\overline{x_i}) \in \Typmf$ for any $i$.

  Let $A$ be the integer matrix defining $\mu$ and $\bm a_i = \matac{a_{i1}}{\cdots}{a_{im}}$ be the $i$-th row of $A$.
  Then, for any $\bm p \in \mathbb R^m$, $\mu^*(\overline{x_i})(\bm p) = \overline{x_i}(A \bm p) = \bm a_i \bm p$.
  Thus the function $\mu^*(\overline{x_i})$ is defined by the monomial $\bm y^{\bm a_i} := y_1^{a_{i1}} \cdots y_m^{a_{im}}$.
  Hence $\mu^*(\overline{x_i}) \in \Typmf$.
\end{proof}

\begin{lem}
  \label{translation}
  Let $\bm a = \matca{a_1}{\vdots}{a_n} \in \mathbb R^n$ be a vector.
  Then the map
  $$\varphi : \Txpmf \to \Txpmf, \qquad f(x_1, \ldots, x_n) \mapsto f(x_1 \odot a_1, \ldots, x_n \odot a_n)$$
  is a semiring isomorphism.
\end{lem}

\begin{proof}
  Clearly $\varphi$ is a semiring homomorphism.
  The map
  $$\psi : \Txpmf \to \Txpmf, \qquad f(x_1, \ldots, x_n) \mapsto f(x_1 \odot (-a_1), \ldots, x_n \odot (-a_n))$$
  is the inverse map of $\varphi$.
  Hence $\varphi$ is an isomorphism.
\end{proof}

In the standard notation, the map in the previous lemma is $f(x_1, \ldots, x_n) \mapsto f(x_1 + a_1, \ldots, x_n + a_n)$.
Thus we call this map the \textit{translation map}.

The following congruences are the main objects of this paper.

\begin{dfn}
  For a subset $Z \subset \mathbb R^n$, we define the congruence $\mathbf E(Z)$ on $\Txpmf$ as
  $$\mathbf E(Z) = \{ (f,g) \ | \ f|_Z = g|_Z \}.$$
\end{dfn}

Note that, for any $\bm a \in \mathbb R^n$, $\mathbf E(Z)$ is finitely generated if and only if $\mathbf E(Z + \bm a)$ is finitely generated, where $Z + \bm a = \{ \bm p + \bm a \ | \ \bm p \in Z \}$.
This is because the translation map is an isomorphism.

\subsection{Tropical varieties}

An \textit{affine linear function} on $\mathbb R^n$ is a function of the form
$$f(x_1, \ldots, x_n) = a_1 x_1 + \cdots + a_n x_n  + b,$$
where $a_1, \ldots, a_n, b \in \mathbb R$ and $(a_1, \ldots, a_n) \neq (0,\ldots, 0)$.
A \textit{polyhedron} in $\mathbb R^n$ is a subset of the form
$$\left\{ \bm p \in \mathbb R^n \ \middle| \ \begin{aligned}
  f_1(\bm p) &\geq 0\\
  &\vdots \\
  f_m(\bm p) &\geq 0
\end{aligned} \right\},$$
where each $f_i$ is an affine linear function on $\mathbb R^n$.
A \textit{polytope} is a bounded polyhedron.

For a polyhedron $\sigma$ in $\mathbb R^n$, let $\Aff (\sigma)$ be the smallest affine linear subspace of $\mathbb R^n$ containing $\sigma$.
The \textit{dimension} $\dim(\sigma)$ of $\sigma$ is the dimension of $\Aff(\sigma)$.
We denote $L(\sigma)$ the linear subspace obtained by translating $\Aff(\sigma)$.

A subset $\tau$ of a polyhedron $\sigma$ is a \textit{face} of $\sigma$ if there exists an affine linear function $f$ such that, for any $\bm p \in \sigma$,
\begin{enumerate}[(1)]
  \item $f(\bm p) \geq 0$, and
  \item $f(\bm p) = 0$ if and only if $\bm p \in \tau$.
\end{enumerate}
Any face of a polyhedron is also a polyhedron.
If $\tau$ is a face of a polyhedron $\sigma$, we denote $\tau \preceq \sigma$.
We also denote $\tau \prec \sigma$ when $\tau \preceq \sigma$ and $\tau \neq \sigma$.
A \textit{polyhedral complex} in $\mathbb R^n$ is a collection $X$ of finitely many polyhedra in $\mathbb R^n$ such that
\begin{enumerate}[(1)]
  \item If $\sigma \in X$ and $\tau \prec \sigma$, then $\tau \in X$, and
  \item If $\sigma, \tau \in X$ and $\sigma \cap \tau \neq \emptyset$, then $\sigma \cap \tau$ is a common face of $\sigma$ and $\tau$. 
\end{enumerate}
In this paper, we assume that every polyhedral complex is \textit{connected}, namely, for any polyhedral complex $X$ and any $\sigma,\tau \in X$, we assume that there exists a sequence $\sigma = \sigma_0, \sigma_1, \ldots, \sigma_k = \tau$ of members of $X$ such that $\sigma_i \prec \sigma_{i+1}$ or $\sigma_i \succ \sigma_{i+1}$ for any $i = 1, \ldots, k-1$.
The \textit{support} of a polyhedral complex $X$ is the union of all members of $X$, which is denoted by $|X|$.
A \textit{dimension} of a polyhedral complex is the maximum dimension of its members.
For a $d$-dimensional polyhedral complex $X$ and $i=1, \ldots, d$, we denote $X_i$ the set of $i$-dimensional polyhedra in $X$.
A \textit{facet} of a polyhedral complex $X$ is a maximal polyhedron in $X$ with respect to inclusion.
A polyhedral complex is \textit{pure dimensional} if its facets have common dimension.
A \textit{weighted polyhedral complex} is a pair $(X, \omega_X)$ of a pure $d$-dimensional polyhedral complex $X$ and a function $\omega_X : X_d \to \mathbb Z_{>0}$.

An affine linear function $f$ is \textit{rational} if $f$ is of the form
$$f(x_1, \ldots, x_n) = a_1 x_1 + \cdots + a_n x_n  + b$$
for some $a_1, \ldots, a_n \in \mathbb Z$ and $b \in \mathbb R$.
A polyhedron is \textit{rational} if it is defined by rational affine linear functions.
A polyhedral complex is \textit{rational} if it consists of rational polyhedra.
For a rational polyhedron $\sigma \subset \mathbb R^n$, we denote $\Lambda(\sigma) = L(\sigma) \cap \mathbb Z^n$, which is a free abelian group of rank $\dim(\sigma)$.

We define a \textit{balanced polyhedral complex} in $\mathbb R^n$ as follows:
Let $(X, \omega_X)$ be a weighted pure $d$-dimensional rational polyhedral complex in $\mathbb R^n$.
For each facet $\sigma$ of $X$ and its $(d-1)$-dimensional face $\tau$, fix a vector $\bm u_{\sigma/\tau} \in \mathbb Z^n$ such that
\begin{enumerate}[ \ (1)]
  \item $\sigma \subset \tau + \mathbb R_{\geq 0} \bm u_{\sigma/\tau}$, and
  \item $\Lambda(\sigma) = \Lambda(\tau) + \mathbb Z \bm u_{\sigma/\tau},$
\end{enumerate}
where the sum in the first condition is the Minkowski sum.
Thus $(X, \omega_X)$ is \textit{balanced} if for any $(d-1)$-dimensional polyhedron $\tau$ in $X$,
$$\sum_{\sigma \succ \tau} \omega_X(\sigma) \bm u_{\sigma / \tau} \in \Lambda(\tau).$$

In this paper, we define a $d$-dimensional \textit{tropical variety} in $\mathbb R^n$ as a pure $d$-dimensional balanced rational polyhedral complex in $\mathbb R^n$.

The following lemma is well-known.

\begin{lem}
  \label{n in n}
  If $X$ is an $n$-dimensional tropical variety in $\mathbb R^n$, then $|X| = \mathbb R^n$.
\end{lem}

An affine linear subspace $W$ of $\mathbb R^n$ is \textit{rational} if it is defined by finitely many rational affine linear functions.
Or equivalently, $W$ is rational if
$$W = \{ \bm p + c_1 \bm q_1 + \cdots c_d \bm q_d \in \mathbb R^n \ | \ c_1, \ldots, c_d \in \mathbb R \}$$
for some $\bm p \in \mathbb R^n$ and $\bm q_1, \ldots, \bm q_n \in \mathbb Z^n$.

Let $W$ be a $d$-dimensional rational affine linear subspace of $\mathbb R^n$.
Fix a vector $\bm p \in W$.
Then $W - \bm p$ is a linear subspace of $\mathbb R^n$.
Since $W$ is rational, $(W - \bm p) \cap \mathbb Z^n$ is a free abelian group of rank $d$.
Let $\{ \bm q_1, \ldots, \bm q_d \}$ be a basis of $(W - \bm p) \cap \mathbb Z^n$.
The map
$$\iota : \mathbb R^d \to W, \qquad \matca{c_1}{\vdots}{c_d} \mapsto \bm p + c_1 \bm q_1 + \cdots c_d \bm q_d$$
is bijective.
Moreover, it induces an isomorphism $\mathbb Z^d \to \mathbb Z \bm q_1 + \cdots + \mathbb Z \bm q_d$.
Hence the following lemma holds.

\begin{lem}
  \label{trop var reduction}
  In the above settings, let $X$ be a tropical variety in $\mathbb R^n$ such that $|X| \subset W$.
  Then the family
  $$\iota^{-1}(X) := \{ \iota^{-1}(\sigma) \ | \ \sigma \in X \}$$
  forms a tropical variety in $\mathbb R^d$.
  Here, for a facet $\sigma \in X$, the weight of $\iota^{-1}(\sigma)$ is that of $\sigma$.
\end{lem}

\subsection{Newton polytopes}

A subset $P \subset \mathbb R^n$ is \textit{convex} if for any $\bm p, \bm q \in P$ and a real number $t$ with $0 \leq t \leq 1$, $t \bm p + (1-t) \bm q \in P$.
For a subset $S \subset \mathbb R^n$, the \textit{convex hull} of $S$ is the smallest convex set containing $S$, which is denoted by $\conv(S)$.
If $S$ is a finite set, $\conv(S)$ is a polytope.

For a tropical Laurent polynomial $P = \bigoplus_{\bm u \in \mathbb Z^n} a_{\bm u} \bm x^{\bm u} \in \Txpm$, the \textit{newton polytope} of $P$ is the polytope
$$\Newt(P) = \conv\{ \bm u \in \mathbb Z^n \ | \ a_{\bm u} \neq -\infty \},$$
which is a \textit{lattice polytope}, i.e., $\Newt(P)$ is a polytope whose vertices are in $\mathbb Z^n$.
It is known that if polynomials $P,Q \in \Txpm$ define the same function, then $\Newt(P) = \Newt(Q)$.
Hence, for a function $f \in \Txpmf$, $\Newt(f)$ is well-defined.
It is easy to check that, for any $f, g \in \Txpmf$,
$$\Newt(f \oplus g) = \conv(\Newt(f) \cup \Newt(g)),$$
and
$$\Newt(f \odot g) = \Newt(f) + \Newt(g).$$
Hence the following holds.

\begin{lem}
  \label{newt in newt}
  Let $h \in \Txpmf$ be any function.
  Assume that
  $$h = (f_{1,1} \odot \cdots \odot f_{1,k_1}) \oplus \cdots \oplus (f_{l,1} \odot \cdots \odot f_{l,k_l}),$$
  for some $f_{i,j} \in \Txpmf$.
  Then each $\Newt(f_{i,j})$ is included in an integer translation of $\Newt(h)$.
\end{lem}

Here, an \textit{integer translation} of a set $Z \subset \mathbb R^n$ is a set of the form $Z + \bm p$ for some $\bm p \in \mathbb Z^n$.

\section{Non-finitely generatedness}

In this section, we show that the congruence $\mathbf E(Z)$ is not finitely generated if $Z$ is a subset of $\mathbb R^n$ satisfying some conditions.

\subsection{A method to show the Non-finitely generatedness}
\label{method}

We introduce some lemmas, and then we consider a way to prove that a given congruence on $\Txpmf$ is not finitely generated.

\begin{lem}[{\cite[Corollary 2.12]{bertram2017tropical}}]
  \label{JM gen}
  Suppose that $E$ is a finitely generated congruence on a semiring $R$, say $E = \lrangle{(a_1,b_1), \ldots, (a_k, b_k)}$.
  Let $\mathcal S \subset R \times R$ denote the collection of all pairs of the form
  $$\left( \sum_{\bm m,\bm n \in (\mathbb Z_{\geq 0})^k } r_{\bm m, \bm n} \bm a^{\bm m}\bm b^{\bm n}, \sum_{\bm m,\bm n \in (\mathbb Z_{\geq 0})^k } r_{\bm m, \bm n} \bm a^{\bm n}\bm b^{\bm m} \right),$$
  where all but finitely many $r_{\bm m,\bm n} \in R$ are zero.
  Then $E$ consists of all pairs $(f,g)$ for which there exists a finite transitive chain $(f, r_1), (r_1, r_2), \ldots, (r_{n-1}, r_n), (r_n, g) \in \mathcal S$.
\end{lem}

Here, if $\bm m = (m_1, \ldots, m_k)$, then $\bm a^{\bm m}$ means $a_1^{m_1} \cdots a_k^{m_k}$.

\begin{lem}
  \label{newt in newt 2}
  Let $E$ be a finitely generated congruence on $\Txpmf$ and fix a generating set $\{ (f_1,g_1), \ldots, (f_k, g_k) \}$ of $E$.
  Assume that there is a pair $(h_1,h_2) \in E$ such that $h_1 \neq h_2$.
  Then at least one of $\Newt(f_1), \Newt(g_1), \ldots, \Newt(f_k), \Newt(g_k)$ is included in an integer translation of $\Newt(h_1)$.
\end{lem}

\begin{proof}
  Let $\mathcal S$ be the collection of all pairs of the form
  $$\left( \sum_{\bm m,\bm n \in (\mathbb Z_{\geq 0})^k } r_{\bm m, \bm n} \bm f^{\bm m}\bm g^{\bm n}, \sum_{\bm m,\bm n \in (\mathbb Z_{\geq 0})^k } r_{\bm m, \bm n} \bm f^{\bm n}\bm g^{\bm m} \right),$$
  where all but finitely many $r_{\bm m,\bm n} \in \Txpmf$ are $-\infty$.
  Then, by Lemma \ref{JM gen}, there exist $r_1, \ldots, r_n \in \Txpmf$ such that
  $$(h_1, r_1), (r_1, r_2), \ldots, (r_{n-1}, r_n), (r_n, h_2) \in \mathcal S.$$
  Since $h_1 \neq h_2$, we may assume that $h_1 \neq r_1$.
  Let
  $$h_1 = \sum_{\bm m,\bm n \in (\mathbb Z_{\geq 0})^k } r_{\bm m, \bm n} \bm f^{\bm m}\bm g^{\bm n}, \qquad r_1 = \sum_{\bm m,\bm n \in (\mathbb Z_{\geq 0})^k } r_{\bm m, \bm n} \bm f^{\bm n}\bm g^{\bm m}.$$
  Since $h_1 \neq r_1$, there exists $(\bm m, \bm n) \in ((\mathbb Z_{\geq 0})^k)^2 \setminus \{ (\mathbf 0, \mathbf 0) \}$ such that $(r_{\bm m, \bm n}) \neq -\infty$.
  Fix such $(\bm m, \bm n)$.
  We may assume that the first coordinate of $\bm m$ is not zero without loss of generality.
  Then, by Lemma \ref{newt in newt}, $\Newt(f_1)$ is included in a integer translation of $\Newt(h_1)$.
\end{proof}

\begin{rem}
  \label{newt in newt infty}
  It is easily shown that even if $E$ is not finitely generated, the statement similar to Lemma \ref{JM gen} holds: Suppose that $E$ is a congruence on a semiring $R$, and let $\mathcal A$ be a generating set of $E$.
  Let $\mathcal S \subset R \times R$ denote the collection of all pairs of the form
  $$\left( \sum_{\bm m,\bm n \in (\mathbb Z_{\geq 0})^k } r_{\bm m, \bm n} \bm a^{\bm m}\bm b^{\bm n}, \sum_{\bm m,\bm n \in (\mathbb Z_{\geq 0})^k } r_{\bm m, \bm n} \bm a^{\bm n}\bm b^{\bm m} \right)$$
  for some pairs $(a_1, b_1), \ldots, (a_k, b_k) \in \mathcal A$, where all but finitely many $r_{\bm m,\bm n} \in R$ are zero.
  Then $E$ consists of all pairs $(f,g)$ for which there exists a finite transitive chain $(f, r_1), (r_1, r_2), \ldots, (r_{n-1}, r_n), (r_n, g) \in \mathcal S$.

  Hence the statement similar to Lemma \ref{newt in newt 2} also holds: Let $E$ be a congruence on $\Txpmf$ and let $E = \lrangle{\mathcal A}$.
  Assume that there is a pair $(h_1,h_2) \in E$ such that $h_1 \neq h_2$.
  Then there is a pair $(f,g) \in \mathcal A$ such that at least one of $\Newt(f), \Newt(g)$ is included in an integer translation of $\Newt(h_1)$.
  We will use these facts in Section \ref{section min gen}.
\end{rem}

By using Lemma \ref{newt in newt 2}, one may consider the following way to prove a given congruence $E$ on $\Txpmf$ is not finitely generated:
\begin{enumerate}[(1)]
  \item Assume that $E$ is finitely generated and let $E = \lrangle{(f_1, g_1), \ldots, (f_k, g_k)}$.
  \item Find a pair $(h_1, h_2) \in E$ such that $h_1 \neq h_2$ and any integer translation of $\Newt(h_1)$ includes none of $\Newt(f_1), \Newt(g_1), \ldots, \Newt(f_k), \Newt(g_k)$.
  \item Then it contradicts to Lemma \ref{newt in newt 2}, hence $E$ is not finitely generated.
\end{enumerate}

However, this method does not work well in general.
That is, if $f_i = -\infty$ for some $f_i$, then the desired $h_1$ does not exist because $\Newt(f_i) = \emptyset$.
Fortunately, if $E$ is of the form $\mathbf E(Z)$ for some nonempty subset $Z \subset \mathbb R^n$, we can avoid this issue.
Indeed, if $(-\infty, g_i) \in \mathbf E(Z)$, obviously $g_i= -\infty$.
Thus we may remove the pair $(f_i, g_i)$ from the generating set.

There is another problem.
If some $f_i$ is a monomial, then $\Newt(f_i)$ consists of a single point.
Hence an integer translation of $\Newt(h_1)$ includes $\Newt(f_i)$ unless $\Newt(h_1) = \emptyset$.
Therefore $h_1$ must be $-\infty$.
However, if $h_1 = -\infty$ and $E$ is of the form $\mathbf E(Z)$ for some nonempty subset $Z \subset \mathbb R^n$, $(h_1, h_2) \in \mathbf E(Z)$ implies $h_1 = h_2 = -\infty$.
Hence there is no desired pair $(h_1, h_2)$.
In the next subsection, we will show that if $Z$ is \textit{unbounded in all rational directions} (see Definition \ref{ubd def}), we may remove monomials from the generating set.

\subsection{Subsets unbounded in all rational directions}

\begin{dfn}
  \label{ubd def}
  A subset $Z \subset \mathbb R^n$ is \textit{unbounded in all rational directions} if, for any $\bm u \in \mathbb Z^n \setminus \{ \mathbf 0 \}$, the set $\{ \bm p \cdot \bm u \in \mathbb R \ | \ \bm p \in Z \}$ is upper unbounded.
\end{dfn}

Here, the symbol $\cdot$ means the standard inner product on $\mathbb R^n$.

\begin{lem}
  \label{ubd1}
  Let $Z \subset \mathbb R^n$ be a subset unbounded in all rational directions.
  Let $f := \overline{c_1 \odot \bm x^{\bm u_1}}, g := \overline{c_2 \odot \bm x^{\bm u_2}} \in \Txpmf$ be monomial functions such that $\bm u_1 \neq \bm u_2$.
  Then there exist points $\bm p, \bm q \in Z$ such that $f(\bm p)>g(\bm p)$ and $f(\bm q) < g(\bm q)$.
\end{lem}

\begin{proof}
  The set $\{ \bm p \cdot (\bm u_1 - \bm u_2) \in \mathbb R \ | \ \bm p \in Z \}$ is upper unbounded.
  Hence there exists $\bm p \in Z$ such that $\bm p \cdot (\bm u_1 - \bm u_2) > c_2 - c_1$.
  This $\bm p$ satisfies $f(\bm p) > g(\bm p)$.
  The existence of $\bm q$ is similarly shown.
\end{proof}

\begin{lem}
  \label{ubd2} 
  Let $Z \subset \mathbb R^n$ be a subset unbounded in all rational directions, and let $(f,g) \in \mathbf E(Z)$ be any pair.
  If $f$ is a monomial function, then $f=g$.
\end{lem}

\begin{proof}

  Let $f = \overline {c \odot \bm x^{\bm u}}$ and $g = \overline{\bigoplus_{i=1}^m c_i \odot \bm x^{\bm v_i}}$.
  We may assume that $\bm v_1, \ldots, \bm v_m$ are distinct.
  If some $\bm v_i$ is not equal to $\bm u$, by Lemma \ref{ubd1}, there exists $\bm p \in Z$ such that
  $$f(\bm p) < c_i + \bm v_i \cdot \bm p \leq g(\bm p),$$
  which contradicts to $(f,g) \in \mathbf E(Z)$.
  Hence $m=1$ and $\bm v_1 = \bm u$.

  Take any point $\bm p \in Z$.
  Since $f(\bm p) = g(\bm p)$,
  $$c + \bm u \cdot \bm p = c_1 + \bm u \cdot \bm p,$$
  which means that $c = c_1$, hence $f=g$.
\end{proof}

\subsection{A condition for not being finitely generated}

Now, we use the method established in Section \ref{method} to show the following theorem.

\begin{thm}
  \label{not fg}
  Assume that $n \geq 2$.
  Let $Z$ be a proper closed subset of $\mathbb R^n$ that is unbounded in all rational directions.
  Then $\mathbf E(Z)$ is not finitely generated.
\end{thm}

The following lemma is helpful.

\begin{lem}
  \label{any polytope}
  Let $Z \subset \mathbb R^n$ be a closed subset such that $\mathbf 0 \not\in Z$.
  Let $P$ be an $n$-dimensional lattice polytope in $\mathbb R^n$ having an interior lattice point.
  Then there exists a pair $(f,g) \in \mathbf E(Z)$ such that $f \neq g$ and $\Newt(f) = P$.
\end{lem}

\begin{proof}
  Let $\bm u_1, \ldots, \bm u_r \in \mathbb Z^n$ be the vertices of $P$ and let $\bm u_0$ be an interior lattice point of $P$.
  Let
  $$f = \overline{\bm x^{\bm u_1} \oplus \cdots \oplus \bm x^{\bm u_r}}.$$
  Also, for each real number $\varepsilon > 0$, consider the function
  $$g_{\varepsilon} = \overline{\varepsilon \odot \bm x^{\bm u_0}} \oplus f.$$
  Then $\Newt(f) = P$.
  Since $f(\mathbf 0) = 0$ and $g_{\varepsilon}(\mathbf 0) = \varepsilon$, we have $f \neq g_{\varepsilon}$.
  Thus it is sufficient to show that there exists $\varepsilon > 0$ such that $(f, g_{\varepsilon}) \in \mathbf E(Z)$.

  Let
  $$D_{\varepsilon} = \{ \bm p \in \mathbb R^n \ | \ \varepsilon + \bm u_0 \cdot \bm p > f(\bm p) \}.$$
  Since $\mathbf 0 \in D_{\varepsilon}$, $D_{\varepsilon}$ is not empty.
  Explicitly, $D_{\varepsilon}$ is the solution set of the inequalities
  $$(\bm u_1 - \bm u_0) \cdot \bm p < \varepsilon,$$
  $$(\bm u_2 - \bm u_0) \cdot \bm p < \varepsilon,$$
  $$\vdots$$
  $$(\bm u_r - \bm u_0) \cdot \bm p < \varepsilon.$$
  Note that the set $\{ \bm u_1 - \bm u_0, \ldots, \bm u_r - \bm u_0 \}$ is the vertex set of the polytope $P + (- \bm u_0)$ and $\mathbf 0$ is an interior point of $P + (- \bm u_0)$.
  Hence $D_{\varepsilon}$ is bounded (see Lemma \ref{app1}).

  For any positive real number $k$, the set $D_{k \varepsilon}$ is a dilation of $D_{\varepsilon}$ with respect to $\mathbf 0$ by a factor of $k$.  
  Since $Z$ is closed and $\mathbf 0 \not\in Z$, there exists a sufficiently small $\varepsilon > 0$ such that $D_{\varepsilon} \cap Z = \emptyset$.
  For such $\varepsilon$, we have $(f, g_{\varepsilon}) \in \mathbf E(Z)$.
  Indeed, for any $\bm p \in Z$, we have $\bm p \not\in D_{\varepsilon}$.
  This means that $f(\bm p) \geq \varepsilon + \bm u_0 \cdot \bm p$.
  Hence
  $$g_{\varepsilon}(\bm p) = \max\{ \varepsilon + \bm u_0 \cdot \bm p, \ f(\bm p) \} = f(\bm p).$$
  Therefore $g_{\varepsilon}|_Z = f|_Z$.
\end{proof}

Next, we construct a ``very thin'' polytope.
Assume that $n \geq 2$.
Fix a positive integer $N$.
Consider the following $n \times n$ matrix:
$$M(n,N) = \begin{pmatrix}
  N-1 & N^2 & N^3 & N^4 & \cdots & N^n \\
  1 & N+1 & 0 & 0 & \cdots & 0 \\
  0 & 0 & 1 & 0 & \cdots & 0 \\
  0 & 0 & 0 & 1 & \cdots & 0 \\
  \vdots & \vdots & \vdots & \vdots && \vdots \\
  0 & 0 & 0 & 0 & \cdots & 1 
\end{pmatrix}$$
Note that $M(n,N)$ is an integer matrix with $\det M(n,N) = -1$.
Hence the map $\varphi : \mathbb R^n \to \mathbb R^n, \ \bm p \mapsto M(n,N) \bm p$ is a linear automorphism and it induces an automorphism on $\mathbb Z^n$.

Let $\bm d_i$ be the $i$-th column of $M(n,N)$, and let $\{ \bm e_1, \cdots, \bm e_n \}$ be the standard basis of $\mathbb R^n$.
Let $P_1$, $P_2$ be the polytopes in $\mathbb R^n$ defined as follows:
$$P_1 = \{ \lambda_1 \bm d_1 + \cdots \lambda_n \bm d_n \in \mathbb R^n \ | \ 0 \leq \lambda_i \leq 2 \text{ for any } i \}.$$
$$P_2 = \{ \lambda_1 \bm e_1 + \cdots \lambda_n \bm e_n \in \mathbb R^n \ | \ 0 \leq \lambda_i \leq 2 \text{ for any } i \}.$$
Via the transformation $\varphi$, $P_2$ corresponds to $P_1$.
Observe that $P_2$ has an interior point at $(1, \ldots, 1)$.
Hence $P_1$ also has an interior point.
Moreover, observe that the set of lattice points of $P_2$ is
$$\{ \lambda_1 \bm e_1 + \cdots \lambda_n \bm e_n \in \mathbb R^n \ | \ \lambda_i \in \{ 0,1,2 \} \text{ for any } i \}.$$
Hence the set of lattice points of $P_1$ is
$$\{ \lambda_1 \bm d_1 + \cdots \lambda_n \bm d_n \in \mathbb R^n \ | \ \lambda_i \in \{ 0,1,2 \} \text{ for any } i \}.$$

\begin{lem}
  Assume that $N \geq 3$.
  For any two distinct lattice points $\bm p, \bm q$ of $P_1$, we have
  $$\norm {\bm p - \bm q} \geq N-2,$$
  where the norm is the Euclidean norm.
\end{lem}

\begin{proof}
  Let $\bm p = \matca{p_1}{\vdots}{p_n}$ and $\bm q = \matca{q_1}{\vdots}{q_n}$.
  Note that
  $$\norm{\bm p - \bm q} \geq |p_1 - q_1|.$$
  By the above observation,
  $$\bm p = \lambda_1 \bm d_1 + \cdots \lambda_n \bm d_n$$
  and
  $$\bm q = \mu_1 \bm d_1 + \cdots \mu_n \bm d_n$$
  for some $\lambda_1, \ldots, \lambda_n, \mu_1, \ldots, \mu_n \in \{0, 1, 2\}$.
  Then
  $$p_1 - q_1 = (\lambda_1 - \mu_1)(N-1) + (\lambda_2- \mu_2)N^2 + \cdots (\lambda_n - \mu_n)N^n.$$
  Since $\bm p \neq \bm q$, $\lambda_i \neq \mu_i$ for some $i$.
  Let $i_0$ be the smallest $i$ such that $\lambda_i \neq \mu_i$.

  If $i_0 = 1$,
  $$\begin{aligned}
    |p_1 - q_1| &\geq |(\lambda_1 - \mu_1)N + (\lambda_2- \mu_2)N^2 + \cdots + (\lambda_n - \mu_n)N^n| - |\lambda_1 - \mu_1| \\
    &\geq N |(\lambda_1 - \mu_1) + (\lambda_2- \mu_2)N + \cdots + (\lambda_n - \mu_n)N^{n-1}| - 2.\\
  \end{aligned}$$
  Note that the integer
  $$N' := (\lambda_1 - \mu_1) + (\lambda_2- \mu_2)N + \cdots + (\lambda_n - \mu_n)N^{n-1}$$
  is not divided by $N$ because $N \geq 3$ and $0 < |\lambda_1 - \mu_1| \leq 2$.
  In particular, $N' \neq 0$.
  Hence $N |N'| - 2 \geq N-2$.

  If $i_0 > 1$,
  $$\begin{aligned}
    |p_1 - q_1| &= | (\lambda_{i_0}- \mu_{i_0})N^{i_0} + \cdots +  (\lambda_n - \mu_n)N^n| \\
    &= N^{i_0} |(\lambda_{i_0} - \mu_{i_0}) + \cdots + (\lambda_n - \mu_n)N^{n-i_0}|.\\
  \end{aligned}$$
  The integer
  $$N'' := (\lambda_{i_0} - \mu_{i_0}) + \cdots + (\lambda_n - \mu_n)N^{n-i_0}$$
  is not $0$ by the similar argument to the above case.
  Hence $N^{i_0}|N''| \geq N > N-2$.
\end{proof}

The following corollary immediately follows.

\begin{cor}
  \label{slender}
  Assume that $n \geq 2$.
  Let $N$ be a positive integer.
  Then there exists an $n$-dimensional lattice polytope $P$ in $\mathbb R^n$ such that
  \begin{enumerate}[(1)]
    \item $P$ has an interior point, and
    \item for any distinct lattice points $\bm p, \bm q \in P$, $\norm{\bm p - \bm q} > N$.     
  \end{enumerate}
\end{cor}

In the next proof, we use the following notation:
For a polytope $P \subset \mathbb R^n$,
$$D(P) := \max_{\bm p, \bm q \in P} \norm{\bm p - \bm q},$$
which is called the \textit{diameter} of $P$.

\begin{proof}[Proof of Theorem \ref{not fg}]
  By Lemma \ref{translation}, we may assume that $\mathbf 0 \not\in Z$.
  Note that the assumption that $Z$ is a proper closed subset of $\mathbb R^n$ that is unbounded in all rational directions is kept by any translation.

  Assume that $\mathbf E(Z)$ is finitely generated, say $\mathbf E(Z) = \lrangle{(f_1,g_1), \ldots, (f_k,g_k)}$.
  If $f_i = -\infty$ for some $i$, then $g_i |_Z = -\infty$, and hence $g_i= -\infty$.
  Then we may remove the pair $(f_i,g_i)$ from the generating set.
  Also, if $f_i$ is a monomial, by Lemma \ref{ubd2}, we have $f_i = g_i$.
  Hence we may remove the pair $(f_i,g_i)$ from the generating set.
  Therefore we may assume that each $f_i, g_i$ is neither $-\infty$ nor a monomial.
  It means that each $\Newt(f_i), \Newt(g_i)$ has at least two distinct lattice points.

  Let $N$ be a positive integer such that
  $$N > \max \{ D(\Newt(f_1)), D(\Newt(g_1)), \ldots, D(\Newt(f_k)), D(\Newt(g_k)) \}.$$
  By Lemma \ref{slender}, there exists an $n$-dimensional lattice polytope $P \subset \mathbb R^n$ such that
  \begin{enumerate}[(1)]
    \item $P$ has an interior point, and
    \item for any distinct lattice points $\bm p, \bm q \in P$, $\norm{\bm p - \bm q} > N$.     
  \end{enumerate}
  By Lemma \ref{any polytope}, there exists a pair $(h_1, h_2) \in \mathbf E(Z)$ such that $h_1 \neq h_2$ and $\Newt(h_1) = P$.
  Hence, by Lemma \ref{newt in newt 2}, at least one of $\Newt(f_1), \Newt(g_1), \ldots, \Newt(f_k), \Newt(g_k)$ is included in an integer translation of $P$.
  This contradicts to the condition (2) and the definition of $N$.
\end{proof}

\section{Reductions}

For a polyhedral complex $X$ in $\mathbb R^n$, we write as $\mathbf E(X) := \mathbf E(|X|)$.
Our next main result is the following.

\begin{thm}
  \label{main}
  Let $X$ be a tropical variety in $\mathbb R^n$.
  Then $\mathbf E(X)$ is finitely generated if and only if $|X|$ is a rational affine linear subspace of $\mathbb R^n$.
\end{thm}

One may consider that we can use Theorem \ref{not fg} in the proof of the only if part.
However, in general, $|X|$ is not unbounded in all rational directions.
In particular, if $|X|$ is included in an proper rational affine linear subspace of $\mathbb R^n$, then clearly $|X|$ is not unbounded in all rational directions.
In such cases, we need to reduce the ambient space.
Note that we may assume that $\mathbf 0 \in \Aff(|X|)$ (i.e. $\Aff(|X|)$ is a linear subspace of $\mathbb R^n$) because the translation map is an isomorphism.

Let $W$ be a $d$-dimensional rational linear subspace of $\mathbb R^n$.
Let $\{ \bm q_1, \ldots, \bm q_d \}$ be a basis of $W \cap \mathbb Z^n$.
The map
$$\iota : \mathbb R^d \to W, \qquad \matca {c_1}{\vdots}{c_d} \mapsto c_1\bm q_1 + \cdots c_d \bm q_d$$
is a linear isomorphism.
Let $\Typmf = \mathbb T[y_1^{\pm}, \ldots, y_d^{\pm}]_{\mathrm{fcn}}$.
Then $\iota$ induces the pull-back map $\iota^* : \Txpmf \to \Typmf, \ f \mapsto f \circ \iota$ (see Lemma \ref{pull-buck}).

\begin{lem}
  The map $\iota^*$ is surjective.
\end{lem}

\begin{proof}
  The basis $\{ \bm q_1, \ldots, \bm q_d \}$ of $W \cap \mathbb Z^n$ can be extended to a basis $\{ \bm q_1, \ldots, \bm q_n \}$ of $\mathbb Z^n$ (see Lemma \ref{app2}).
  We now regard vectors in $\mathbb R^n$ as column vectors.
  Consider the matrix $A = \matac{\bm q_1}{\cdots}{\bm q_n}$.
  Then $\det A = \pm 1$, and hence there is the inverse $A^{-1}$, which is an integer matrix.
  Let $\bm u_i$ be the $i$-th row of $A^{-1}$ for each $i$.
  Since $A^{-1}A=E$, we have $\bm u_i \bm q_j = \delta_{ij}$, where $\delta_{ij}$ is the Kronecker delta.
  Let us compute $\iota^*(\overline{\bm x^{\bm u_i}})$ for $i=1, \ldots, d$:
  $$\begin{aligned}
    \iota^*(\overline{\bm x^{\bm u_i}})(\bm p) &= \overline{\bm x^{\bm u_i}}(\iota(\bm p)) \\
    &= \overline{\bm x^{\bm u_i}}(\matac{\bm q_1}{\cdots}{\bm q_d}\bm p)\\
    &= \bm u_i \matac{\bm q_1}{\cdots}{\bm q_d}\bm p\\
    &= {}^t\!\bm e_i \bm p\\
    &= (\text{the $i$-th entry of $\bm p$}),
  \end{aligned}$$
  where $\{ \bm e_1, \ldots \bm e_d \}$ is the standard basis of $\mathbb R^d$.
  Hence $\iota^*(\overline{\bm x^{\bm u_i}}) = \overline{y_i}$.
  This also means $\iota^*(\overline{\bm x^{-\bm u_i}}) = \overline{y_i^{-1}}$.
  Therefore $\iota^*$ is surjective.
\end{proof}

We recall the notion of push-forwards of congruences.
Let $\varphi : R_1 \to R_2$ be a semiring homomorphism and let $E$ be a congruence on $R_1$.
We denote
$$\varphi(E) = \{ (\varphi(f), \varphi(g)) \in R_2^2 \ | \ (f,g) \in E \}.$$
The set $\varphi(E)$ is not a congruence on $R_2$ in general.
We denote $\varphi_*(E)$ the congruence generated by $\varphi(E)$.

\begin{lem}[{\cite[Corollary 2.15]{bertram2017tropical}}]
  \label{BE cor 2.15}
  Suppose $\varphi : R_1 \to R_2$ is a surjective semiring morphism and $E$ is a congruence on $R_1$.
  If $E = \lrangle{(a_1,b_1),\ldots, (a_k, b_k)}$, then
  $$\varphi_*(E) = \lrangle{(\varphi(a_1), \varphi(b_1)), \ldots, (\varphi(a_k), \varphi(b_k))}.$$
\end{lem}

\begin{lem}
  \label{push EiZ}
  Let $Z \subset W$ be any subset.
  Then $(\iota^*)_*(\mathbf E(Z)) = \mathbf E(\iota^{-1}(Z))$.
\end{lem}

\begin{proof}
  To show the inclusion $\subset$, it is enough to show that $\iota^*(\mathbf E(Z)) \subset \mathbf E(\iota^{-1}(Z))$.
  Let $(f,g) \in \mathbf E(Z)$ be any pair and let $\bm p \in \iota^{-1}(Z)$ be any point.
  Then
  $$\iota^*(f)(\bm p) = f(\iota(\bm p)) = g(\iota(\bm p)) = \iota^*(g)(\bm p).$$
  Hence $(\iota^*(f), \iota^*(g)) \in \mathbf E(\iota^{-1}(Z))$, which means that $\iota^*(\mathbf E(Z)) \subset \mathbf E(\iota^{-1}(Z))$.

  To show the converse, take any pair $(f,g) \in \mathbf E(\iota^{-1}(Z))$.
  Since $\iota^*$ is surjective, there exist $f_1, g_1 \in \Txpmf$ such that $\iota^*(f_1) = f$ and $\iota^*(g_1) = g$.
  We claim that $(f_1, g_1) \in \mathbf E(Z)$.
  Indeed, let $\bm p \in Z$ be any point and take $\bm q \in \iota^{-1}(Z)$ such that $\iota(\bm q) = \bm p$.
  Since $(f,g) \in \mathbf E(\iota^{-1}(Z))$, $f(\bm q) = g(\bm q)$.
  Thus
  $$f_1(\bm p) = f(\bm q) = g(\bm q) = g_1(\bm p).$$
  Hence $(f_1, g_1) \in \mathbf E(Z)$.
  Therefore
  \[ (f,g) = (\iota^*(f_1), \iota^*(g_1)) \in \iota^*(\mathbf E(Z)) \subset (\iota^*)_*(\mathbf E(Z)). \qedhere \]
\end{proof}

\begin{rem}
  The above proof also shows that $\iota^*(\mathbf E(Z)) = \mathbf E(\iota^{-1}(Z))$.
  In particular, $\iota^*(\mathbf E(Z))$ is already a congruence.
\end{rem}

\begin{lem}
  \label{fg reduction}
  Let $Z \subset W$ be any subset.
  If $\mathbf E(Z)$ is finitely generated, then $\mathbf E(\iota^{-1}(Z))$ is also finitely generated.
\end{lem}

\begin{proof}
  This follows from Lemma \ref{BE cor 2.15} and Lemma \ref{push EiZ}.
\end{proof}

Hence, for a subset $Z \subset W$, in order to show that $\mathbf E(Z)$ is not finitely generated, it is enough to show that $\mathbf E(\iota^{-1}(Z))$ is not finitely generated.
Thus we may assume that $Z$ is not included in any proper rational affine linear subspace of $\mathbb R^n$.
Note that if $Z$ is the support of a tropical variety in $\mathbb R^n$, then so is $\iota^{-1}(Z)$ by Lemma \ref{trop var reduction}.

\section{Proof of Theorem \ref{main}}

\subsection{Proof of the if part of Theorem \ref{main}}

Let $X$ be a tropical variety in $\mathbb R^n$.
Assume that $W := |X|$ is a $d$-dimensional rational affine linear subspace of $\mathbb R^n$.
By Lemma \ref{translation}, we may assume that $\mathbf 0 \in W$, i.e., $W$ is a linear subspace of $\mathbb R^n$.
Since $W$ is rational, the orthogonal complement $W'$ (with respect to the standard inner product) of $W$ is also rational.
We use the notation $\Lambda(W') = W' \cap \mathbb Z^n$.
Let $\{ \bm q_1, \ldots, \bm q_{n-d} \}$ be a basis of $\Lambda(W')$.
We see that the congruence $\mathbf E(W)$ is generated by the finite set $S := \left\{ \left(\overline{\bm x^{\bm q_i}}, 0 \right) \ | \ i=1, \ldots, n-d \right\}$.

Note that by Lemma \ref{app2}, $\Lambda(W')$ is a direct summand of $\mathbb Z^n$.
Thus we fix a subgroup $H$ of $\mathbb Z^n$ with $\mathbb Z^n = H \oplus \Lambda(W')$.

For any $\bm p \in W$ and $i = 1, \ldots, n-d$,
$$\overline{\bm x^{\bm q_i}}(\bm p) = \bm q_i \cdot \bm p = 0.$$
Hence $S \subset \mathbf E(W)$, and then $\lrangle{S} \subset \mathbf E(W)$.
Conversely, take any $(f,g) \in \mathbf E(W)$.
Let
$$f = \overline{\bigoplus_{\bm u} a_{\bm u}{\bm x}^{\bm u}}.$$
Fix any $\bm u \in \mathbb Z^n$.
Since $\mathbb Z^n = H \oplus \Lambda(W')$, $\bm u = \bm p + \bm q$ for some $\bm p \in H$ and $\bm q \in \Lambda(W')$.
Let $\bm q = c_1 \bm q_1 + \cdots + c_{n-d} \bm q_{n-d} \ (c_1, \ldots, c_{n-d} \in \mathbb Z)$.
Note that
$$\overline{\bm x^{\bm q_i}} \sim_{\lrangle{S}} 0$$
for any $i$.
Hence
$$\overline{\bm x^{\bm u}} = \overline{\bm x^{\bm p} \odot \bm x^{c_1\bm q_1 + \cdots + c_{n-d} \bm q_{n-d}}} \sim_{\lrangle{S}} \overline{\bm x^{\bm p}}.$$
By applying this argument for all $\bm u$, it is shown that there exists a function $f_0 \in \Txpmf$ such that $f \sim_{\lrangle{S}} f_0$ and $f_0$ is of the form
$$\overline{\bigoplus_{\bm p \in H} a_{\bm p} \bm x^{\bm p}}.$$
Similarly, there exists a function $g_0 \in \Txpmf$ such that $g \sim_{\lrangle{S}} g_0$ and $g_0$ is of the form
$$\overline{\bigoplus_{\bm p \in H} b_{\bm p} \bm x^{\bm p}}.$$
Since $\lrangle{S} \subset \mathbf E(W)$ and $(f,g) \in \mathbf E(W)$, we have
$$f_0|_W = f|_W = g|_W = g_0|_W.$$

Now, we show that $f_0=g_0$ as functions on $\mathbb R^n$.
To see this, let $V$ be the linear subspace of $\mathbb R^n$ generated by $H$, and $V'$ the orthogonal complement of $V$.
We claim that $\mathbb R^n = W \oplus V'$.
Indeed, since $\mathbb R^n = W' \oplus V$, we have
$$(W + V')^{\perp} = W' \cap V = \{ \mathbf 0 \}$$
and
$$(W \cap V')^{\perp} = W' + V = \mathbb R^n,$$
where $\square^{\perp}$ means the orthogonal complement of $\square$.
These mean that $W+V' = \mathbb R^n$ and $W \cap V' = \{ \mathbf 0 \}$, hence $\mathbb R^n = W \oplus V'$.

Take any $\bm v \in \mathbb R^n$.
Then $\bm v = \bm v_1 + \bm v_2$ for some $\bm v_1 \in W$ and $\bm v_2 \in V'$.
Note that, for any $\bm p \in H$,
$$\overline{\bm x^{\bm p}}(\bm v) = \bm p \cdot (\bm v_1 + \bm v_2) = \bm p \cdot \bm v_1 = \overline{\bm x^{\bm p}}(\bm v_1).$$
Hence
$$f_0(\bm v) = f_0(\bm v_1) = g_0(\bm v_1) = g_0(\bm v).$$
Here, the second equality holds because $f_0|_W = g_0|_W$.
Therefore $f_0 = g_0$ as functions on $\mathbb R^n$.
Thus $f \sim_{\lrangle{S}} f_0 = g_0 \sim_{\lrangle{S}} g$, which means that $\mathbf E(W) \subset \lrangle{S}$.
This completes the proof of the if part of Theorem \ref{main}.

\subsection{Proof of the only if part of Theorem \ref{main}}

We need the following proposition.

\begin{prop}
  \label{trop unbounded}
  Let $X$ be a tropical variety in $\mathbb R^n$.
  Assume that $|X|$ is not included in a proper rational affine linear subspace of $\mathbb R^n$.
  Then $|X|$ is unbounded in all rational directions.
\end{prop}

It is sufficient to show the following.

\begin{prop}
  \label{trop unbounded in f}
  Let $X$ be a tropical variety in $\mathbb R^n$ and $\bm p \in \mathbb Z^n \setminus \{ \mathbf 0 \}$.
  Assume that the function $f: |X| \to \mathbb R, \bm a \mapsto \bm a \cdot \bm p$ is not constant.
  Then $f$ is upper unbounded.
\end{prop}

\begin{proof}[Proof of \textquotedblleft Proposition \ref{trop unbounded in f} $\Longrightarrow$ Proposition \ref{trop unbounded}\textquotedblright]
  Assume that Proposition \ref{trop unbounded in f} holds.
  Let $X$ be a tropical variety in $\mathbb R^n$ whose support is not included in a proper rational affine linear subspace of $\mathbb R^n$.
  Take any $\bm p \in \mathbb Z^n \setminus \{ \mathbf 0 \}$.
  If the function $f: |X| \to \mathbb R, \bm a \mapsto \bm a \cdot \bm p$ is constant with the value $k$, then $|X|$ is included in the set $\{ \bm a \in \mathbb R^n \ | \ \bm a \cdot \bm p = k \}$, which is a proper rational affine linear subspace of $\mathbb R^n$.
  This contradicts to the assumption, hence $f$ is not constant.
  Thus $f$ is upper unbounded.
  Since $\bm p$ is an arbitrary element of $\mathbb Z^n \setminus \{ \mathbf 0 \}$, $|X|$ is unbounded in all rational directions.
\end{proof}

First we show Proposition \ref{trop unbounded in f} in the case $\dim X=1$.
We use the following terminology:
Let $\bm p$ be a vertex of $X$, $\sigma$ a 1-dimensional face of $X$ adjacent to $\bm p$, and $\bm d$ the primitive direction vector of $\sigma$ such that $\bm p + t \bm d \in \sigma$ for sufficiently small $t>0$.
Then we say that $\bm d$ is the primitive direction vector of $\sigma$ \textit{from} $\bm p$.

\begin{lem}
  \label{1dim balance}
  Let $X = (X, \omega_X)$ be a 1-dimensional tropical variety in $\mathbb R^n$, and $\bm p$ a vertex of $X$.
  Let $\sigma_1, \ldots, \sigma_k$ be the 1-dimensional faces of $X$ adjacent to $\bm p$.
  For each $i=1, \ldots, k$, let $\bm d_i$ be the primitive direction vector of $\sigma_i$ from $\bm p$.
  Let $\bm q \in \mathbb Z^n \setminus \{ \mathbf 0 \}$.
  Assume that there exists an $i$ such that $\bm d_i \cdot \bm q < 0$.
  Then there exists a $j$ such that $\bm d_j \cdot \bm q > 0$.
\end{lem}

\begin{proof}
  Since $X$ is balanced at $\bm p$,
  $$\sum_{m=1}^k \omega_X(\sigma_m) \bm d_m = \mathbf 0.$$
  Hence
  $$\sum_{m=1}^k \omega_X(\sigma_m)\bm d_m \cdot \bm q = 0.$$
  By the assumption, one of $\bm d_m \cdot \bm q$ is negative.
  Thus one of $\bm d_m \cdot \bm q$ is positive.
\end{proof}

\begin{prop}
  \label{1dim ubd}
  Let $X = (X, \omega_X)$ be a 1-dimensional tropical variety in $\mathbb R^n$ and $\bm p \in \mathbb Z^n \setminus \{ \mathbf 0 \}$.
  Assume that the function $f: |X| \to \mathbb R, \bm a \mapsto \bm a \cdot \bm p$ is not constant on $|X|$.
  Then $f$ is upper unbounded on $|X|$.
\end{prop}

\begin{proof}
  First we claim that $X$ has a 1-dimensional face $\sigma$ on which $f$ is not constant.
  By the assumption, $f$ is not constant on $|X|$.
  Then there exist $\bm a, \bm b \in |X|$ such that $f(\bm a) < f(\bm b)$
  Since $X$ is connected, there exists a sequence $\sigma_1, \ldots, \sigma_k$ of 1-dimensional face of $X$ such that $\bm a \in \sigma_1, \bm b \in \sigma_k$, and for any $i=1, \ldots, k-1$, $\sigma_i$ and $\sigma_{i+1}$ are adjacent.
  If $f$ is constant on $\sigma_1 \cup \ldots \cup \sigma_k$, then $f(\bm a) = f(\bm b)$, which is a contradiction.
  Hence $f$ is not constant on some $\sigma_i$.

  Let $\sigma'_0$ be a 1-dimensional face of $X$ on which $f$ is not constant.
  Take the primitive direction vector $\bm d_0$ of $\sigma'_0$ with $\bm d_0 \cdot \bm p > 0$.
  Take any interior point $\bm a_0 \in \sigma'_0$.
  If $\sigma'_0$ includes the ray $\bm a_0 + \mathbb R_{\geq 0}\bm d_0$, then $f$ is upper unbounded on $\sigma'_0$, and hence the proof is completed.
  Otherwise, let $\bm a_1$ be the endpoint of $\sigma'_0$ on which $f|_{\sigma'_0}$ takes the maximum value.
  The primitive direction vector of $\sigma'_0$ from $\bm a_1$ is $-\bm d_0$.
  By Lemma \ref{1dim balance} and $(-\bm d_0) \cdot \bm p < 0$, there exists a 1-dimensional face $\sigma'_1$ of $X$ such that its primitive direction vector $\bm d_1$ from $\bm a_1$ satisfies $\bm d_1 \cdot \bm p > 0$.
  If $\sigma'_1$ includes the ray $\bm a_1 + \mathbb R_{\geq 0}\bm d_1$, then $f$ is upper unbounded on $\sigma'_0$, and hence the proof is completed.
  Otherwise, let $\bm a_2$ be the endpoint of $\sigma'_1$ on which $f|_{\sigma'_1}$ takes the maximum value, and iterate the same argument.
  Since $X$ consists of finitely many polyhedra, eventually this iteration will stop.\qedhere

\end{proof}

Next we consider the case $\dim X \geq 2$.
We briefly recall the notion of stable intersections.
See \cite{maclagan2015introduction} for detail.

Let $X,Y$ be tropical varieties in $\mathbb R^n$.
Then the \textit{stable intersection} $X \cap_{st} Y$ of $X$ and $Y$ is the (probably non-connected) polyhedral complex
$$\{ \sigma_1 \cap \sigma_2 \ | \ \sigma_1 \in X, \ \sigma_2 \in Y, \ \dim (\sigma_1 + \sigma_2) = n \}$$
endowed with a certain weight function.
We use the following lemma.

\begin{lem}[{\cite[Theorem 3.6.10]{maclagan2015introduction}}]
  \label{stable insec}
  If $X$ and $Y$ are tropical varieties in $\mathbb R^n$ and $\dim X= d_1, \dim Y = d_2$, then $X \cap_{st} Y$ is empty or a (probably non-connected) pure $(d_1 + d_2 - n)$-dimensional weighted balanced rational polyhedral complex.
  In particular, each connected component of $X \cap_{st} Y$ is a tropical variety.
\end{lem}

\begin{proof}[Proof of Proposition \ref{trop unbounded} and Proposition \ref{trop unbounded in f}]
  Let $d = \dim X$.
  By Proposition \ref{1dim ubd}, the statement holds when $d=1$.

  Assume that $d \geq 2$.
  By the argument similar to the proof of Proposition \ref{1dim ubd}, there exists a face $\sigma \in X$ on which $f$ is not constant.
  We may assume that $\dim (\sigma) = d$.
  Indeed, if otherwise, replace $\sigma$ by a facet of $X$ including it.
  Take a rational line segment $l$ in $\sigma$ on which $f$ is not constant.
  Thus there exists a $(n-d+1)$-dimensional rational affine linear subspace $W$ such that $l \subset W$ and $\dim (\sigma \cap W) = 1$.
  Indeed, take any point $\bm p \in l$ and let $\bm d_1 \subset \mathbb Z^n$ be a direction vector of $l$.
  Take a basis $\{ \bm d_1, \ldots, \bm d_d \}$ of $L(\sigma)$ consisting of integer vectors, and then extend it to a basis $\{ \bm d_1, \ldots, \bm d_n \}$ of $\mathbb R^n$ consisting of integer vectors.
  Thus the affine linear subspace $W = \bm p + \lrangle{\bm d_1,\bm d_{d+1}, \ldots \bm d_n}$ works.
  Note that the singleton $\{ W \}$ forms a tropical variety in $\mathbb R^n$ with the weight $\omega(W)=1$.

  Let $Y$ be a connected component of $X \cap_{st} \{ W \}$ whose support includes $l$.
  By Lemma \ref{stable insec}, $Y$ is a 1-dimensional tropical variety.
  Since $l \subset |Y|$, $f$ is not constant on $|Y|$.
  Hence, by Proposition \ref{1dim ubd}, $f$ is upper unbounded on $|Y|$, which means that $f$ is upper unbounded on $|X|$.

  Proposition \ref{trop unbounded} follows from Proposition \ref{trop unbounded in f}.  \qedhere
\end{proof}

\begin{proof}[Proof of the only if part of Theorem \ref{main}]
  Let $X$ be a tropical variety such that $|X|$ is not a rational affine linear subspace of $\mathbb R^n$.
  We may assume that $|X|$ is not included in any proper rational affine linear subspace of $\mathbb R^n$ by Lemma \ref{fg reduction} and the observation after it.
  Hence, by Proposition \ref{trop unbounded}, $|X|$ is unbounded in all rational directions.
  In order to use Theorem \ref{not fg}, we check that $n \geq 2$.
  If $n =0$, $|X| = \{ \mathbf 0 \}$.
  It contradicts to the assumption that $|X|$ is not an affine linear subspace.
  If $n =1$, $|X|$ is a singleton or the whole space $\mathbb R$ (by Lemma \ref{n in n}).
  It also contradicts to the assumption.
  Thus $n \geq 2$, and hence by Theorem \ref{not fg}, $\mathbf E(X)$ is not finitely generated. \qedhere
\end{proof}

\section{An example of a minimal generating set}

\label{section min gen}

This section is independent of Section 3-5 except for that we will use Remark \ref{newt in newt infty} in the proof of minimality.
In this section, we describe a minimal generating set of the congruence $\mathbf E(L)$ explicitly, where $L$ is the \textit{standard tropical line} defined as follows:
Let $\{ \bm e_1, \bm e_2 \}$ be the standard basis of $\mathbb R^2$.
Let $\rho_1, \rho_2, \rho_3$ be the rays spanned by $-\bm e_1, -\bm e_2, \bm e_1 + \bm e_2$ respectively.
Then we define $L = \{ \rho_1, \rho_2, \rho_3, \{ \mathbf 0 \} \}$.
The weight of any ray is 1.
For simplicity, for any functions $f,g \in \Txpmf$, we write $f|_L=g|_L$ instead of $f|_{|L|} = g|_{|L|}$.

We introduce some notations.
Let $A = \{ (x_1, y_1), \ldots, (x_k, y_k) \}$ be a finite subset of $\mathbb Z^2$.
Let $a = \min \{ x_1, \ldots, x_k \}$, \ $b = \min \{ y_1, \ldots, y_k \}$, and $c = \max \{ x_1+y_1, \ldots, x_k+y_k \}$.
We denote
$$\Delta(A) = \{ (x,y) \in \mathbb Z^2 \ | \ x \geq a, \ y \geq b, \ x+y \leq c \}.$$
For a finite subset $A \subset \mathbb Z^2$, we define the tropical Laurent polynomial $P_A$ as
$$P_A = \bigoplus_{\bm u \in A} \bm x^{\bm u},$$
and the function $f_A$ as $f_A = \overline{P_A}$.
Now, let
$$\mathcal S = \left\{ \left( \overline{0 \oplus x^u y^v}, f_{\Delta(\{ (0,0), (u,v) \})} \right) \ \middle| \ \begin{aligned}
  &\text{$u$ and $v$ are integers with $\gcd(u,v) = 1$,}\\
  &\text{ $u > 0$ or $(u,v) = (0,1)$}
\end{aligned} \right\}.$$
We show that $\mathcal S$ is a minimal generating set of $\mathbf E(L)$.

\begin{lem}
  For any finite subsets $A,B \subset \mathbb Z^2$, $f_A|_L = f_B|_L$ if and only if $\Delta(A) = \Delta(B)$.
\end{lem}

\begin{proof}
  Let $A = \{ (x_i,y_i) \}_{i=1}^l, \ B = \{ (z_j,w_j) \}_{j=1}^k$.
  For any $t \geq 0$,
  $$f_A(-t\bm e_1) = \max\{ -t x_1, \ldots, -t x_l \} = -t \min \{ x_1 , \ldots, x_l \},$$
  and
  $$f_B(-t\bm e_1) = \max\{ -t z_1, \ldots, -t z_k \} = -t \min \{ z_1 , \ldots, z_k \}.$$
  Hence the equality $f_A|_{\rho_1} = f_B|_{\rho_1}$ holds if and only if $\min \{ x_1 , \ldots, x_l \} = \min \{ z_1 , \ldots, z_k \}$.
  Similarly, $f_A|_{\rho_2} = f_B|_{\rho_2}$ (resp. $f_A|_{\rho_3} = f_B|_{\rho_3}$) holds if and only if $\min \{ y_1 , \ldots, y_l \} = \min \{ w_1 , \ldots, w_k \}$ (resp. $\max \{ x_1+y_1 , \ldots, x_l+y_l \} = \max \{ z_1+w_1 , \ldots, z_k+w_k \}$).
  Thus $f_A|_L = f_B|_L$ if and only if $\Delta(A) = \Delta(B)$.
\end{proof}

\begin{cor}
  For any finite subsets $A \subset \mathbb Z^2$, $f_A|_L = f_{\Delta(A)}|_L$.
\end{cor}

\begin{proof}
  This follows from $\Delta(\Delta(A)) = \Delta(A)$.
\end{proof}

This corollary shows that $\mathcal S \subset \mathbf E(L)$.
Hence $\lrangle{S} \subset \mathbf E(L)$.

To show the converse inclusion, we introduce some more terms.
Let $P \in \Txpm$ be a tropical Laurent polynomial and $a \bm x^{\bm u}$ a term of it.
Let $P_{\check{\bm u}}$ be the polynomial obtained by removing the term $a \bm x^{\bm u}$ from $P$.
Then the term $a \bm x^{\bm u}$ of $P$ is \textit{effective on} $L$ if $\overline{P}|_L \neq \overline{P_{\check{\bm u}}}|_L$.

Consider a tropical Laurent polynomial of the form
$$0 \oplus \bigoplus_{i=1}^{\infty} a_{-i}x^{-i} \oplus \bigoplus_{j=1}^{\infty} b_{-j}y^{-j} \oplus \bigoplus_{k=1}^{\infty} c_{k}x^{k},$$
where $a_{-i} < 0, \ b_{-j} < 0$, $c_k \leq 0$, and $a_{-i} = b_{-j} = c_k = - \infty$ except for finitely many $i,j,k$.
We call it a polynomial of the \textit{standard form}.
In addition, if every term is effective on $L$, we call it a polynomial of the \textit{standard form effective on} $L$.
Note that for any polynomial $P$ of the standard form, by removing non-effective terms, we obtain a polynomial of the standard form effective on $L$.
Later we describe a complete representative system of $\Txpmf / \mathbf E(L)$ by using polynomials of the standard form.

Let $\rho \in L$ be a ray and $\bm d$ the primitive direction vector of it.
For any function $f \in \Txpmf \setminus \{ -\infty \}$, there is a sufficiently small $\varepsilon > 0$ such that the function $\varphi : [0, \varepsilon] \to \mathbb R, \ t \mapsto f(t \bm d)$ is affine linear.
We call the slope of $\varphi$ the \textit{starting slope of} $f$ \textit{on} $\rho$.

In order to describe a complete representative system of $\Txpmf / \mathbf E(L)$, we introduce some lemmas.

\begin{lem}
  \label{sum of slopes}
  Let $f \in \Txpmf$ be a function, and $m_i$ the starting slope of $f$ on $\rho_i$ \ ($i=1,2,3$).
  Then $m_1 + m_2 + m_3 \geq 0$.
\end{lem}

\begin{proof}
  Let $f = \overline{P}$ and $f(\mathbf 0) = a$.
  Thus $P$ has a term of the form $a x^uy^v$.
  The starting slope of $\overline{a x^uy^v}$ on $\rho_i$ is less than or equal to $m_i$.
  The starting slope of $\overline{a x^uy^v}$ on $\rho_1$ (resp. $\rho_2$, resp. $\rho_3$) is $-u$ (resp. $-v$, resp. $u+v$).
  Hence $m_1 + m_2 + m_3 \geq (-u) + (-v) + u+v = 0$.
\end{proof}

\begin{lem}
  \label{standard equiv}
  Let $f \in \Txpmf$.
  Then the following are equivalent:
  \begin{enumerate}[(1)]
    \item There exists a polynomial $P$ of the standard from such that $f|_L = \overline P|_L$,
    \item $f(\mathbf 0) = 0$, and both the starting slopes of $f$ on $\rho_1, \rho_2$ are $0$.
  \end{enumerate} 
\end{lem}

\begin{proof}
  (1) $\Longrightarrow$ (2): Let
  $$P = 0 \oplus \bigoplus_{i=1}^{\infty} a_{-i}x^{-i} \oplus \bigoplus_{j=1}^{\infty} b_{-j}y^{-j} \oplus \bigoplus_{k=1}^{\infty} c_{k}x^{k}.$$
  Since $a_{-i}, b_{-j}, c_k \leq 0$, $f(\mathbf 0) = 0$.
  For a sufficiently small $t>0$,
  $$f(-t\bm e_1) = \max \{ 0, \ \max\{ a_{-i} + it \ | \ i \geq 1 \}, \ \max\{ b_{-j} \ | \ j \geq 1 \}, \ \max\{ c_{k} - kt \ | \ k \geq 1\}  \} = 0.$$
  Hence the starting slope of $f$ on $\rho_1$ is 0.
  Similarly the starting slope of $f$ on $\rho_2$ is also 0.

  (2) $\Longrightarrow$ (1): The function $\varphi_1 : [0, \infty) \to \mathbb R, \ t \mapsto f(-t \bm e_1)$ is a tropical Laurent polynomial function.
  Since $f(\mathbf 0) = 0$ and the starting slope of $f$ on $\rho_1$ is 0, $\varphi$ is defined by a tropical polynomial of the form
  $$0 \oplus a_{-1}t \oplus a_{-2} t^2 + \cdots a_{-n_1}t^{n_1},$$
  where $a_{-i} < 0$.
  Similarly, the function $\varphi_2 : [0, \infty) \to \mathbb R, \ t \mapsto f(-t \bm e_2)$ is the function defined by a tropical polynomial of the form
  $$0 \oplus b_{-1}t \oplus b_{-2} t^2 + \cdots b_{-n_2}t^{n_2},$$
  where $b_{-j}<0$, and the function $\varphi_3 : [0, \infty) \to \mathbb R, \ t \mapsto f(t(\bm e_1 + \bm e_2))$ is the function defined by a tropical polynomial of the form
  $$0 \oplus c_1t \oplus c_{2} t^2 + \cdots c_{n_3}t^{n_3},$$
  where $c_k \leq 0$.
  Here, we use the fact that the starting slope of $f$ on $\rho_3$ is nonnegative, which follows from the previous lemma.
  By using these $a_{-i}, b_{-j}, c_k$, let
  $$P = 0 \oplus \bigoplus_{i=1}^{n_1} a_{-i}x^{-i} \oplus \bigoplus_{j=1}^{n_2} b_{-j}y^{-j} \oplus \bigoplus_{k=1}^{n_3} c_{k}x^{k}.$$
  Note that the term of the form $b_{-j} y^{-j}$ or $c_k x^k$ takes a nonpositive value on $\rho_1$.
  Hence
  $$\overline{P}|_{\rho_1} = \left. \overline{0 \oplus \bigoplus_{i=1}^{\infty} a_{-i}x^{-i}} \right|_{\rho_1} = f|_{\rho_1}.$$
  Similarly, $\overline{P}|_{\rho_2} = f|_{\rho_2}$ and $\overline{P}|_{\rho_3} = f|_{\rho_3}$.
  Therefore $\overline{P}|_L = f|_L$.
\end{proof}

\begin{lem}
  Let $P,Q$ be tropical Laurent polynomials of the standard forms effective on $L$.
  If $\overline{P}|_L = \overline{Q}|_L$, then $P=Q$.
\end{lem}

\begin{proof}
  Assume that $P$ has a term of the form $a x^{-i}$ for some $a < 0$ and a positive integer $i$.
  Since the term $a x^{-i}$ of $P$ is effective on $L$, there exists a point $\bm p \in L$ such that $\overline{a x^{-i}}(\bm p) > \overline{P_{\check{(-i,0)}}}(\bm p)$.
  By the continuity, $\overline{a x^{-i}} > \overline{P_{\check{(-i,0)}}}$ on a neighborhood $U$ of $\bm p$.
  Note that $a x^{-i}$ takes a negative value on a point in $\rho_2 \cup \rho_3$.
  Thus $\overline{a x^{-i}}|_{\rho_2 \cup \rho_3} \leq 0 \leq \overline{P}|_{\rho_2 \cup \rho_3}$.
  Hence we have $U \cap L \subset \rho_1$.
  
  Consider the function $\varphi : \mathbb R_{\geq 0} \to \mathbb R, \ t \mapsto P(-t\bm e_1)$.
  By the above argument, $\varphi(t) = a + it$ on some interval $I \subset \mathbb R_{\geq 0}$.
  Since $\overline{P}|_L = \overline{Q}|_L$, the function $\psi : \mathbb R_{\geq 0} \to \mathbb R, \ t \mapsto Q(-t\bm e_1)$ satisfies $\psi(t) = a+it$ on $I$.
  This function cannot be realized by the monomial of the form $b y^{-j} \ (b<0, j \in \mathbb Z_{>0})$ or $c x^k \ (c \leq 0, k \in \mathbb Z_{>0})$.
  Thus $Q$ has a term of the form $a'x^{-i'}$ such that $\overline{a'x^{-i'}}(t) = a + it$ for $t \in I$.
  The equality $\overline{a'x^{-i'}}(t) = a + it$ means $a' + i't = a+it$.
  Hence $a' = a$ and $i'=i$.
  Therefore $Q$ has the term $a x^{-i}$.

  The similar argument works for the terms of the form $b y^{-v}$ and $c x^k$.
  Hence every term of $P$ appears in $Q$.
  By the same argument, every term of $Q$ appears in $P$.
  Hence $P=Q$.
\end{proof}

\begin{cor}
  \label{coincide}
  Let $f,g \in \Txpmf$ be functions defined by tropical Laurent polynomials of the standard forms.
  If $f|_L = g|_L$, then $f=g$.
\end{cor}

\begin{proof}
  Take tropical Laurent polynomials $P,Q$ of the standard forms such that $f = \overline{P}$ and $g = \overline{Q}$.
  By removing non-effective terms, we may assume tha both $P$ and $Q$ are effective on $L$.
  Hence, by the previous lemma, we have $P=Q$, and then $f = g$.
\end{proof}

The following lemma gives a complete representative system of $\Txpmf / \mathbf E(L)$.

\begin{lem}
  \label{standard decom}
  For any $f \in \Txpmf \setminus \{ -\infty \}$, there exists a unique tuple $(a,u,v,f_0)$ consisting of $a \in \mathbb R$, \ $u, v \in \mathbb Z$, and a function $f_0 \in \Txpmf$ defined by a tropical Laurent polynomial of the standard form such that $f|_L = \overline{ax^uy^v}f_0|_L$.
\end{lem}

\begin{proof}
  We first show the uniqueness.
  Assume that $f|_L = \overline{ax^uy^v}f_0|_L$ for some $a \in \mathbb R, u,v \in \mathbb Z$ and a function $f_0$ defined by a tropical Laurent polynomial of the standard form.
  By substituting $\bm x = \mathbf 0$, we have $a = f(\mathbf 0)$.
  The starting slope of $\overline{ax^uy^v}f_0$ on $\rho_1$ (resp. $\rho_2$) is $u$ (resp. $v$).
  Hence $u,v$ is determined by the function $f$.
  Therefore $a, u$ and $v$ are unique.
  Since $f_0|_L = \left( \overline{ax^uy^v} \right)^{-1} f |_L$, $f_0|_L$ is also uniquely determined by $f$.
  By Corollary \ref{coincide}, $f_0$ is unique.

  We show the existence.
  Let $a = f(\mathbf 0)$, and let $u$ (resp. $v$) be the starting slope of $f$ on $\rho_1$ (resp. $\rho_2$).
  Then $\left( \overline{ax^uy^v} \right)^{-1} f$ satisfies the condition (2) in Lemma \ref{standard equiv}.
  Hence there exists a function $f_0$ defined by a tropical Laurent polynomial of the standard form such that $\left( \overline{ax^uy^v} \right)^{-1} f|_L = f_0|_L$.
  Thus $f|_L = \overline{ax^uy^v}f_0|_L$.
\end{proof}

Next we see some properties of $\lrangle{\mathcal S}$.
Note that for any finite subsets $A,B \subset \mathbb Z^2$, clearly
$$f_A \odot f_B = f_{A + B},$$
where $A + B = \{ \bm a + \bm b \ | \ \bm a \in A, \ \bm b \in B \}$.
Also note that for any finite subsets $A,B \subset \mathbb Z^2$, $\Delta(A + B) = \Delta(A) + \Delta(B)$.
This follows from the definition of $\Delta$.

\begin{lem}
  Let $A \subset \mathbb Z^2$ be a finite subset.
  \begin{enumerate}[(1)]
    \item For any $\bm u \in \mathbb Z^2$, $\overline{\bm x^{\bm u}} \odot f_A = f_{A + \{ \bm u \}}$.
    \item For any positive integer $n$, $(f_{\Delta(A)})^n = f_{\Delta(nA)}$, where $nA = \{ n \bm a \ | \ \bm a \in A \}$.
  \end{enumerate}
\end{lem}

\begin{proof}
  (1) is clear.
  (2) holds because
  $$\begin{aligned}
    (f_{\Delta(A)})^n &= f_{\Delta(A) + \cdots + \Delta(A)}\\
    &= f_{\Delta(A + \cdots + A)}\\
    &= f_{\Delta(nA)},
  \end{aligned}$$
  where the equality $\Delta(A + \cdots + A) = \Delta(nA)$ follows from the definition of $\Delta$.
\end{proof}

\begin{lem}
  \label{binom Delta}
  For any nonempty finite subset $A \subset \mathbb Z^2$,
  $$f_A \sim_{\lrangle{\mathcal S}} f_{\Delta(A)}.$$
\end{lem}

\begin{proof}
  If $|A| = 1$, it is clear.
  Assume that $|A| = 2$, say $A = \{ \bm u_1, \bm u_2 \}$.
  Let $\bm u_1 - \bm u_2 = (gu,gv)$, where $g$ is a positive integer and $u,v$ are integers with $\gcd(u,v)=1$.
  If $u > 0$ or $(u,v) = (0,1)$, then the pair $\left( \overline{0 \oplus x^uy^v}, f_{\Delta(\{ (0,0),(u,v) \})} \right)$ is in $\mathcal S$.
  Thus
  $$\begin{aligned}
    \overline{\bm x^{\bm u_1} \oplus \bm x^{\bm u_2}} &= \overline{\bm x^{\bm u_2}} \odot \left( \overline{\bm x^{\bm u_1 - \bm u_2} \oplus 0} \right)\\
    &= \overline{\bm x^{\bm u_2}} \odot \left( \overline{x^uy^v \oplus 0} \right)^g\\
    &\sim_{\lrangle{\mathcal S}} \overline{\bm x^{\bm u_2}} \odot \left( f_{\Delta(\{ (0,0), (u,v) \})} \right)^g\\
    &= f_{\Delta(\{ (0,0), (gu, gv) \} + \{ \bm u_2 \})}\\
    &= f_{\Delta(\{ \bm u_2, \bm u_1 \})}.
  \end{aligned}$$
  If $(u,v)$ satisfies neither $u > 0$ nor $(u,v) = (0,1)$, by exchanging $\bm u_1$ and $\bm u_2$, the same argument works.
  Hence the case $|A| = 2$ is shown.

  We now show in the general case.
  Let $A = \{ (u_1, v_1), \ldots, (u_k, v_k) \}$, \ $a = \min \{ u_1, \ldots, u_k \}$, \ $b = \min \{ v_1, \ldots, v_k \}$, and $c = \max \{ u_1+v_1, \ldots, u_k+v_k \}$.
  Thus there exist $\bm u_1 = (u_{i_1},v_{i_1}), \bm u_2 = (u_{i_2},v_{i_2}), \bm u_3 = (u_{i_3},v_{i_3}) \in A$ (probably not distinct) such that $u_{i_1} =a, v_{i_2} = b$, and $u_{i_3} + v_{i_3} = c$.
  Note that $(a,b) \in \Delta(\{ \bm u_1, \bm u_2 \})$ and $\Delta(\{ (a,b), \bm u_3 \}) = \Delta(A)$ by the definition of $\Delta$.
  Hence we have
  $$\begin{aligned}
    f_A &= f_A \oplus \overline{\bm x^{\bm u_1} \oplus \bm x^{\bm u_2}} \oplus \overline{\bm x^{\bm u_3}}\\
    &\sim_{\lrangle{\mathcal S}} f_A \oplus f_{\Delta(\{ \bm u_1, \bm u_2 \})} \oplus \overline{\bm x^{\bm u_3}}\\
    &= f_A \oplus f_{\Delta(\{ \bm u_1, \bm u_2 \})} \oplus \overline{x^ay^b} \oplus \overline{\bm x^{\bm u_3}}\\
    &\sim_{\lrangle{\mathcal S}} f_A \oplus f_{\Delta(\{ \bm u_1, \bm u_2 \})} \oplus f_{\Delta(A)}\\
    &= f_{\Delta(A)},
  \end{aligned}$$
  where, in the last equality, we use that $A \subset \Delta(A)$ and $\Delta(\{ \bm u_1, \bm u_2 \}) \subset \Delta(A)$.
\end{proof}

\begin{lem}
  \label{binom standard}
  For any integers $u,v$, 
  $$\overline{0 \oplus x^uy^v} \sim_{\lrangle{\mathcal S}} \overline{0 \oplus \bigoplus_{i=1}^{\infty} a_{-i}x^{-i} \oplus \bigoplus_{j=1}^{\infty} b_{-j}y^{-j} \bigoplus_{k=1}^{\infty} c_{k}x^{k}}$$
  for some $a_{-i}, b_{-j}, c_k \in \{ 0, -\infty \}$ with $a_{-i} = b_{-j} = c_k = -\infty$ except for finitely many $i,j,k$.
\end{lem}

\begin{proof}
  If $u \geq 0$ and $v \geq 0$, $\Delta(\{ (0,0), (u,v) \}) = \Delta(\{ (0,0), (u+v,0) \})$, see Figure 1 (1).
  Hence, by the previous lemma, 
  $$\overline{0 \oplus x^uy^v} \sim_{\lrangle{\mathcal S}} f_{\Delta(\{ (0,0), (u,v) \})} \sim_{\lrangle{\mathcal S}} \overline{0 \oplus x^{u+v}}.$$
  The other cases are similar.
  See Figure 1 (2)--(6).
  \begin{itemize}
  \item If $u + v \geq 0$ and $v \leq 0$, then $\overline{0 \oplus x^uy^v} \sim_{\lrangle{\mathcal S}} \overline{0 \oplus y^v \oplus x^{u+v}}$.
  \item If $u + v \leq 0$ and $u \geq 0$, then $\overline{0 \oplus x^uy^v} \sim_{\lrangle{\mathcal S}} \overline{0 \oplus y^v}.$
  \item If $u \leq 0$ and $v \geq 0$, then $\overline{0 \oplus x^uy^v} \sim_{\lrangle{\mathcal S}} \overline{0 \oplus x^u \oplus y^v}.$
  \item If $v \geq 0$ and $u+v \leq 0$, then $\overline{0 \oplus x^uy^v} \sim_{\lrangle{\mathcal S}} \overline{0 \oplus x^u}.$
  \item If $u \leq 0$ and $u+v \geq 0$, then $\overline{0 \oplus x^uy^v} \sim_{\lrangle{\mathcal S}} \overline{0 \oplus x^u \oplus x^{u+v}}.$
  \end{itemize}
  Thus the proof is completed.
\end{proof}

\begin{figure}[htbp]
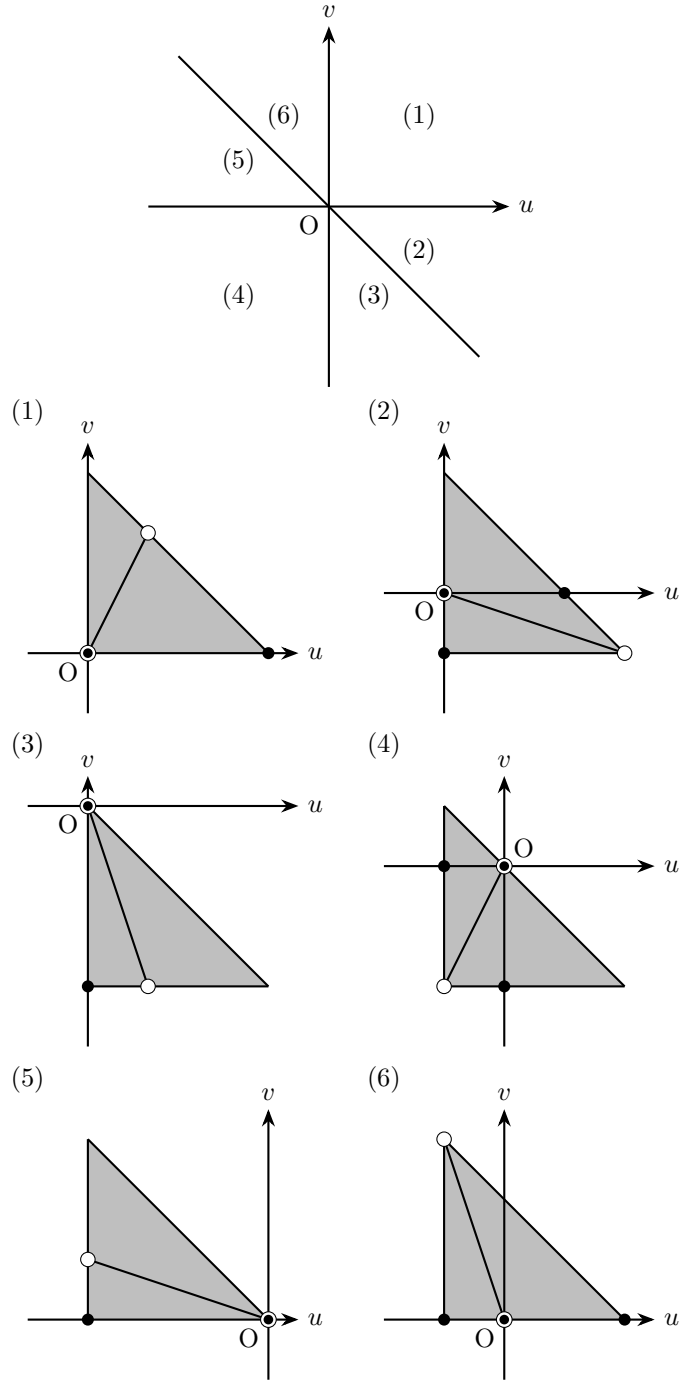

\label{Delta transform}
\centering

\tikz[xscale=0.8, yscale=0.8]{

    \node at (0,0) [anchor=north east] {O};
    \node at (3,0) [anchor=west] {$u$};
    \node at (0,3) [anchor=south] {$v$};

    \draw [thick, -Stealth] (-3,0)--(3,0);
    \draw [thick, -Stealth] (0,-3)--(0,3);
   
    \draw [thick] (-2.5,2.5)--(2.5,-2.5);

    \node at (1.5,1.5) {(1)};
    \node at (1.5,-0.75) {(2)};
    \node at (0.75,-1.5) {(3)};
    \node at (-1.5,-1.5) {(4)};
    \node at (-1.5,0.75) {(5)};
    \node at (-0.75,1.5) {(6)};
   }

\begin{tabular}{cc}
  
  \tikz[xscale=0.8, yscale=0.8]{
    \fill[lightgray] (0,0)--(3,0)--(0,3)--cycle;

    \node at (0,0) [anchor=north east] {O};
    \node at (3.5,0) [anchor=west] {$u$};
    \node at (0,3.5) [anchor=south] {$v$};

    \draw [thick, -Stealth] (-1,0)--(3.5,0);
    \draw [thick, -Stealth] (0,-1)--(0,3.5);
    \node at (-1,4) {(1)};
   
    \draw [thick] (3,0)--(0,3);

    \draw [thick] (0,0)--(1,2);

    \fill[white] (0,0) circle [radius=0.13];
    \draw (0,0) circle [radius=0.13];
    \fill[white] (1,2) circle [radius=0.12];
    \draw (1,2) circle [radius=0.12];
    \fill[black] (0,0) circle [radius=0.08];
    \fill[black] (3,0) circle [radius=0.1];
   }
   &
  
  \tikz[xscale=0.8, yscale=0.8]{
    \fill[lightgray] (0,-1)--(3,-1)--(0,2)--cycle;

    \node at (0,0) [anchor=north east] {O};
    \node at (3.5,0) [anchor=west] {$u$};
    \node at (0,2.5) [anchor=south] {$v$};

    \draw [thick, -Stealth] (-1,0)--(3.5,0);
    \draw [thick, -Stealth] (0,-2)--(0,2.5);
    \node at (-1,3) {(2)};

    \draw [thick] (3,-1)--(0,2);
    \draw [thick] (0,-1)--(3,-1);

    \draw [thick] (0,0)--(3,-1);

    \fill[white] (0,0) circle [radius=0.13];
    \draw (0,0) circle [radius=0.13];
    \fill[white] (3,-1) circle [radius=0.12];
    \draw (3,-1) circle [radius=0.12];
    \fill[black] (0,0) circle [radius=0.08];
    \fill[black] (0,-1) circle [radius=0.1];
    \fill[black] (2,0) circle [radius=0.1];
   }\\

  \tikz[xscale=0.8, yscale=0.8]{
    \fill[lightgray] (0,-3)--(3,-3)--(0,0)--cycle;

    \node at (0,0) [anchor=north east] {O};
    \node at (3.5,0) [anchor=west] {$u$};
    \node at (0,0.5) [anchor=south] {$v$};

    \draw [thick, -Stealth] (-1,0)--(3.5,0);
    \draw [thick, -Stealth] (0,-4)--(0,0.5);
    \node at (-1,1) {(3)};
   
    \draw [thick] (3,-3)--(0,0);
    \draw [thick] (0,-3)--(3,-3);

    \draw [thick] (0,0)--(1,-3);

    \fill[white] (0,0) circle [radius=0.13];
    \draw (0,0) circle [radius=0.13];
    \fill[white] (1,-3) circle [radius=0.12];
    \draw (1,-3) circle [radius=0.12];
    \fill[black] (0,0) circle [radius=0.08];
    \fill[black] (0,-3) circle [radius=0.1];
   } & 

     \tikz[xscale=0.8, yscale=0.8]{
    \fill[lightgray] (-1,-2)--(2,-2)--(-1,1)--cycle;

    \node at (0,0) [anchor=south west] {O};
    \node at (2.5,0) [anchor=west] {$u$};
    \node at (0,1.5) [anchor=south] {$v$};

    \draw [thick, -Stealth] (-2,0)--(2.5,0);
    \draw [thick, -Stealth] (0,-3)--(0,1.5);
    \node at (-2,2) {(4)};
   
    \draw [thick] (2,-2)--(-1,1);
    \draw [thick] (-1,-2)--(2,-2);
    \draw [thick] (-1,1)--(-1,-2);

    \draw [thick] (0,0)--(-1,-2);

    \fill[white] (0,0) circle [radius=0.13];
    \draw (0,0) circle [radius=0.13];
    \fill[white] (-1,-2) circle [radius=0.12];
    \draw (-1,-2) circle [radius=0.12];
    \fill[black] (0,0) circle [radius=0.08];
    \fill[black] (-1,0) circle [radius=0.1];
    \fill[black] (0,-2) circle [radius=0.1];
   }\\

     \tikz[xscale=0.8, yscale=0.8]{
    \fill[lightgray] (-3,0)--(0,0)--(-3,3)--cycle;

    \node at (0,0) [anchor=north east] {O};
    \node at (0.5,0) [anchor=west] {$u$};
    \node at (0,3.5) [anchor=south] {$v$};

    \draw [thick, -Stealth] (-4,0)--(0.5,0);
    \draw [thick, -Stealth] (0,-1)--(0,3.5);
    \node at (-4,4) {(5)};
   
    \draw [thick] (-3,0)--(-3,3);
    \draw [thick] (-3,3)--(0,0);

    \draw [thick] (0,0)--(-3,1);

    \fill[white] (0,0) circle [radius=0.13];
    \draw (0,0) circle [radius=0.13];
    \fill[white] (-3,1) circle [radius=0.12];
    \draw (-3,1) circle [radius=0.12];
    \fill[black] (0,0) circle [radius=0.08];
    \fill[black] (-3,0) circle [radius=0.1];
   } & 

     \tikz[xscale=0.8, yscale=0.8]{
    \fill[lightgray] (-1,0)--(2,0)--(-1,3)--cycle;

    \node at (0,0) [anchor=north east] {O};
    \node at (2.5,0) [anchor=west] {$u$};
    \node at (0,3.5) [anchor=south] {$v$};
    \node at (-2,4) {(6)};

    \draw [thick, -Stealth] (-2,0)--(2.5,0);
    \draw [thick, -Stealth] (0,-1)--(0,3.5);
   
    \draw [thick] (2,0)--(-1,3);
    \draw [thick] (-1,0)--(-1,3);

    \draw [thick] (0,0)--(-1,3);

    \fill[white] (0,0) circle [radius=0.13];
    \draw (0,0) circle [radius=0.13];
    \fill[white] (-1,3) circle [radius=0.12];
    \draw (-1,3) circle [radius=0.12];
    \fill[black] (0,0) circle [radius=0.08];
    \fill[black] (2,0) circle [radius=0.1];
    \fill[black] (-1,0) circle [radius=0.1];
   }

\end{tabular}
  \caption{$\Delta(\text{white circled points}) = \Delta(\text{black circled points})$}
\end{figure}

\begin{lem}
  \label{standard by S}
  Let $f \in \Txpmf$ be a function satisfying the condition (2) in Lemma \ref{standard equiv}.
  Then there exists a function $f_1$ defined by a tropical Laurent polynomial of the standard form such that $f \sim_{\lrangle{\mathcal S}} f_1$.
\end{lem}

\begin{proof}
  Let $f = \overline P, \ P \in \Txpm$.
  We first see that we may assume that $P$ has the term $0$.
  Since $f(\mathbf 0) = 0$, the constant term of $P$ is nonpositive.
  Since $f(\mathbf 0) = 0$ and the starting slope of $f$ on $\rho_1$ is 0, $P$ has a term of the form $y^v$.
  If $v < 0$, $y^v$ takes a positive value on $\rho_2$, which contradicts to the assumption that the starting slope of $f$ on $\rho_2$ is 0.
  Hence $v \geq 0$.
  Similarly, $P$ has a term of the form $x^u$ \ ($u \geq 0$).
  Therefore, by Lemma \ref{binom Delta},
  $$\overline{P} = \overline{P \oplus x^u \oplus y^v} \sim_{\lrangle{\mathcal S}} \overline{P \oplus P_{\Delta(\{ (u,0), (0,v) \})}}.$$
  Here, in the first equality, we use the fact $x^u \oplus x^u = x^u, \ y^v \oplus y^v = y^v$.
  Since $(0,0) \in \Delta(\{ (u,0), (0,v) \})$, $P_{\Delta(\{ (u,0), (0,v) \})}$ has the term 0.
  Thus we may assume that $P$ has the term 0.

  Let $P= \bigoplus_{\bm u \in A} a_{\bm u}\bm x^{\bm u}$, where $A \subset \mathbb Z^2 \setminus \{ (0,0) \}$ is a finite subset including $\mathbf 0$ and $a_{\mathbf 0} = 0$.
  Since $f(\mathbf 0) = 0$, $a_{\bm u} \leq 0$ for any $\bm u \in A$.
  Hence, for any $\bm u \in A$, $0 \oplus a_{\bm u} \bm x^{\bm u} = 0 \oplus a_{\bm u} \odot (0 \oplus \bm x^{\bm u})$.
  This means that
  $$P = 0 \oplus \bigoplus_{\bm u \in A} a_{\bm u}(0 \oplus \bm x^{\bm u}).$$
  By using Lemma \ref{binom standard} for each $0 \oplus \bm x^{\bm u}$, we have
  $$\overline{P} \sim_{\lrangle{\mathcal S}} \overline{0 \oplus \bigoplus_{i=1}^{\infty} a_{-i}x^{-i} \oplus \bigoplus_{j=1}^{\infty} b_{-j}y^{-j} \bigoplus_{k=1}^{\infty} c_{k}x^{k}}$$
  for some $a_{-i}, b_{-j}, c_k \leq 0$ with $a_{-i} = b_{-j} = c_k = -\infty$ except for finitely many $i,j,k$.
  If $a_{-i}=0$ for some $i$, the starting slope of $\overline{a_{-i} x^{-i}}$ on $\rho_1$ is positive, which is a contradiction.
  Hence $a_{-i} < 0$.
  Similarly, $b_{-j} < 0$.
  Therefore the polynomial $0 \oplus \bigoplus_{i=1}^{\infty} a_{-i}x^{-i} \oplus \bigoplus_{j=1}^{\infty} b_{-j}y^{-j} \bigoplus_{k=1}^{\infty} c_{k}x^{k}$ is of the standard form.
\end{proof}

We now show that $\mathbf E(L) = \lrangle{\mathcal S}$.

\begin{prop}
  The congruence $\mathbf E(L)$ is generated by $\mathcal S$.
\end{prop}

\begin{proof}
  The inclusion $\mathbf E(L) \supset \mathcal S$ is already shown.
  Take any pair $(f, g) \in \mathbf E(L)$.
  By Lemma \ref{standard decom}, there exists a unique tuple $(a,u,v,f_0)$ consisting of $a \in \mathbb R$, \ $u, v \in \mathbb Z$, and a function $f_0 \in \Txpmf$ defined by a tropical Laurent polynomial of the standard form such that $f|_L = g|_L = \overline{ax^uy^v}f_0|_L$.
  Since $\overline{(-a)x^{-u}y^{-v}}f|_L=\overline{(-a)x^{-u}y^{-v}}g|_L = f_0|_L$,  both of the functions $\overline{(-a)x^{-u}y^{-v}}f$ and $\overline{(-a)x^{-u}y^{-v}}f$ satisfy the condition (2).  
  Hence, by Lemma \ref{standard by S}, there are functions $f_1, g_1$ defined by tropical Laurent polynomials of the standard forms such that $\overline{(-a)x^{-u}y^{-v}}f \sim_{\lrangle{\mathcal S}} f_1$ and $\overline{(-a)x^{-u}y^{-v}}g \sim_{\lrangle{\mathcal S}} g_1$.
  Since $\lrangle{\mathcal S} \subset \mathbf E(L)$, we have
  $$f_1|_L = \overline{(-a)x^{-u}y^{-v}}f|_L = \overline{(-a)x^{-u}y^{-v}}g|_L = g_1|_L.$$
  Therefore, by Corollary \ref{coincide}, $f_1 = g_1$, which shows that
  $$\overline{(-a)x^{-u}y^{-v}}f \sim_{\lrangle{\mathcal S}} f_1 = g_1 \sim_{\lrangle{\mathcal S}} \overline{(-a)x^{-u}y^{-v}}g.$$
  Thus $f \sim_{\lrangle{\mathcal S}} g$.
\end{proof}

Finally, we show the minimality of $\mathcal S$.

\begin{thm}
  The set $\mathcal S$ is a minimal generating set of $\mathbf E(L)$.
\end{thm}

\begin{proof}
  By the previous proposition, $\mathcal S$ is a generating set of $\mathbf E(L)$.
  Take any pair $\xi = \left( \overline{0 \oplus x^u y^v}, f_{\Delta(\{ (0,0), (u,v) \})} \right) \in \mathcal S$ and let $\mathcal S' = \mathcal S \setminus \{ \xi \}$.
  Assume that $\xi \in \lrangle{\mathcal S'}$.
  The Newton polytope $\Newt(\overline{0 \oplus x^u y^v})$ is the line segment $\conv(\{ (0,0), (u,v) \})$.
  Since $\gcd(u,v) = 1$, the polytopes included in $\Newt(\overline{0 \oplus x^u y^v})$ are $\{ (0,0) \}$, $\{ (u,v) \}$, and $\conv(\{ (0,0), (u,v) \})$ itself.
  Hence, by Remark \ref{newt in newt infty}, there exists a pair $(f,g) \in \mathcal S'$ such that at least one of $\Newt(f), \Newt(g)$ is $\{ (0,0) \}$, $\{ (u,v) \}$, or $\conv(\{ (0,0), (u,v) \})$.
  However, by the definition of $\mathcal S$, such pair does not exist.
  Thus it is a contradiction.
  Therefore $\xi \not \in \lrangle{\mathcal S'}$.
\end{proof}

\begin{rem}
  The existence of a minimal generating set of $\mathbf E(L)$ consisting of infinitely many pairs means that $\mathbf E(L)$ is not finitely generated.
  Namely, if $\mathbf E(L)$ is finitely generated, say $\mathbf E(L) = \lrangle{(f_1, g_1), \ldots, (f_k, g_k)}$, then $(f_i, g_i) \in \lrangle{\mathcal S}$ for each $i$.
  Thus, by Remark \ref{newt in newt infty}, there exist some pairs $\xi_{i1}, \ldots, \xi_{il_i} \in \mathcal S$ such that $(f_i, g_i) \in \lrangle{\xi_{i1}, \ldots, \xi_{il_i}}$.
  Hence
  $$\{ (f_1, g_1), \ldots, (f_k, g_k) \} \subset \lrangle{\xi_{11}, \ldots, \xi_{1l_1}, \ldots, \xi_{k1}, \ldots, \xi_{kl_k}}.$$
  This means that $\mathbf E(L) = \lrangle{\xi_{11}, \ldots, \xi_{1l_1}, \ldots, \xi_{k1}, \ldots, \xi_{kl_k}}$, which contradicts to the minimality of $\mathcal S$.
  Thus the result of this section gives another proof of Theorem \ref{main} in the case $X = L$.
\end{rem}

\setcounter{section}{0}

\renewcommand{\thesection}{\Alph{section}}

\section{Appendix}

\begin{lem}
  \label{app2}
  Let $W$ be a $d$-dimensional rational linear subspace of $\mathbb R^n$, and $\{ \bm q_1, \ldots, \bm q_d \}$ a basis of the free abelian group $W \cap \mathbb Z^n$.
  Then the basis $\{ \bm q_1, \ldots, \bm q_d \}$ can be extended to a basis $\{ \bm q_1, \ldots, \bm q_n \}$ of $\mathbb Z^n$.
\end{lem}

\begin{proof}
  Note that $\mathbb Z^n/ (W \cap \mathbb Z^n)$ is torsion free.
  Indeed, let $\bm a \in \mathbb Z^n$ be a vector and let $m$ be a nonzero integer such that $m \bm a \in W \cap \mathbb Z^n$.
  Since $W$ is a linear subspace of $\mathbb R^n$, $\bm a \in W$, hence $\bm a \in W \cap \mathbb Z^n$.

  Thus, by the fundamental theorem of finitely generated abelian groups, $\mathbb Z^n/ (W \cap \mathbb Z^n)$ is a free abelian group.
  Consider the short exact sequence
  $$0 \longrightarrow W \cap \mathbb Z^n \longrightarrow \mathbb Z^n \longrightarrow \mathbb Z^n/ (W \cap \mathbb Z^n) \longrightarrow 0 $$
  of abelian groups.
  By computing the lengths of these groups, we verify that the rank of $\mathbb Z^n/ (W \cap \mathbb Z^n)$ is $n-d$.
  For any vector $\bm a \in \mathbb Z^n$, we denote $\overline{\bm a}$ the equivalence class of $\bm a$ in $\mathbb Z^n/ (W \cap \mathbb Z^n)$.
  Let $\bm q_{d+1}, \ldots, \bm q_n \in \mathbb Z^n$ be vectors such that $\{ \overline{\bm q_{d+1}}, \ldots, \overline{\bm q_n} \}$ forms a basis of $\mathbb Z^n/ (W \cap \mathbb Z^n)$.
  Then there is the section
  $$\varphi : \mathbb Z^n/ (W \cap \mathbb Z^n) \to \mathbb Z^n, \qquad \overline{\bm q_i} \mapsto \bm q_i \ (i=d+1, \ldots, n).$$
  This shows that $\mathbb Z^n = (W \cap \mathbb Z^n) \oplus \Im(\varphi)$, and then $\{ \bm q_1, \ldots, \bm q_n \}$ forms a basis of $\mathbb Z^n$.
\end{proof}

\begin{lem}
  \label{app1}
  Let $\bm u_1, \ldots, \bm u_k \in \mathbb R^n$ be vectors such that the set $P := \conv\{\bm u_1, \ldots, \bm u_k\}$ is an $n$-dimensional polytope and $\mathbf 0$ is an interior point of $P$.
  Then, for any $a_1, \ldots, a_k \in \mathbb R$, the set
  $$Q = \left\{ \bm q \in \mathbb R^n \ | \ \bm u_i \cdot \bm q < a_i \text{ for any } i \right\}$$
  is bounded (possibly empty).
\end{lem}

\begin{proof}
  Let $\{ \bm e_1, \ldots, \bm e_n \}$ be the standard basis of $\mathbb R^n$.
  Take any $i \in \{ 1, \ldots, n \}$.
  Since $\mathbf 0$ is an interior point of $P$, $t \bm e_i \in P$ for sufficiently small $t > 0$.
  For such $t$, we have
  $$t \bm e_i = s_1 \bm u_1 + \cdots + s_k \bm u_k$$
  for some $s_1, \ldots, s_k \geq 0$ with $(s_1, \ldots, s_k) \neq (0, \ldots, 0)$.
  Hence, for any $\bm q \in Q$,
  $$\begin{aligned}
    \bm e_i \cdot \bm q &= \frac 1t (s_1 \bm u_1 + \cdots + s_k \bm u_k) \cdot \bm q \\
    &= \frac 1t \sum_{j=1}^k s_j (\bm u_j \cdot \bm q)\\
    &< \frac 1t \sum_{j=1}^k s_j a_j =: b_i.
  \end{aligned}$$
  Note that $\bm e_i \cdot \bm q$ is the $i$-th coordinate of $\bm q$.
  Therefore the $i$-th coordinates of vectors in $Q$ are less than $b_i$.

  Since $-t \bm e_i \in P$ for sufficiently small $t > 0$, by the same argument, there exists $c_i \in \mathbb R$ such that
  $$-\bm e_i \cdot \bm q < c_i$$
  for any $\bm q \in Q$.
  Hence the $i$-th coordinates of vectors in $Q$ are greater than $-c_i$.
  Therefore
  $$Q \subset (-c_1,b_1) \times \cdots \times (-c_n, b_n),$$
  which means that $Q$ is bounded.
\end{proof}

\end{document}